\documentclass[sn-mathphys,Numbered]{sn-jnl}%

\usepackage{subfig}

\usepackage{graphicx}%
\usepackage{multirow}%
\usepackage{amsmath,amssymb,amsfonts}%
\usepackage{amsthm}%
\usepackage{mathrsfs}%
\usepackage[title]{appendix}%
\usepackage{xcolor}%
\usepackage{textcomp}%
\usepackage{manyfoot}%
\usepackage{booktabs}%
\usepackage{algorithm}%

\usepackage{listings}%
\newtheorem{theorem}{Theorem}
\newtheorem{proposition}{Proposition}
\newtheorem{assumption}{Assumption}
\usepackage{hyperref}       %

\newtheorem{remark}{Remark}%
\newtheorem{lemma}{Lemma}%
\newtheorem{corollary}{Corollary}%

\newcommand{\sm}[2]{\begin{smallmatrix}\item1\\\item2 \end{smallmatrix}}

\newcommand\ddfrac[2]{\frac{\displaystyle \item1}{\displaystyle \item2}}

\newcommand{\mc}[1]{\mathbb{\item1}}

\RequirePackage{algpseudocode}

\newcommand{\cK}{\mathcal{K}}

\newcommand{\cN}{\mathcal{N}}
\newcommand{\cO}{\mathcal{O}}

\newcommand{\cR}{\mathcal{R}}
\newcommand{\cS}{\mathcal{S}}

\newcommand{\cV}{\mathcal{V}}
\newcommand{\cW}{\mathcal{W}}

\newcommand{\boldm}{\mathbf{m}}
\newcommand{\bp}{\mathbf{p}}
\newcommand{\bx}{\mathbf{x}}
\newcommand{\by}{\mathbf{y}}
\newcommand{\bz}{\mathbf{z}}
\newcommand{\bI}{\mathbf{I}}
\newcommand{\bP}{\mathbf{P}}
\newcommand{\bW}{\mathbf{W}}

\newcommand{\R}{\mathbb{R}}
\newcommand{\simplex}{S_1(d)}
\newcommand{\set}{S}

\newcommand{\boldOne}{\mathbf{1}}

\newcommand{\cRperp}{\cR^{\perp}}
\newcommand{\lmax}{\lambda_{\max}}
\newcommand{\lminp}{\lambda_{\min}^{+}}
\newcommand{\rgam}{_{r,\gamma}}
\newcommand{\eqdef}{:=}
\newcommand{\transpose}{\ensuremath{^{\top}}}
\newcommand{\erdos}{Erd\H{o}s--R\'{e}nyi}

\DeclareMathOperator*{\argmin}{arg\,min}

\begin{document}

\title[Wasserstein Barycenter on Time-Varying Networks]{Decentralized Convex Optimization on Time-Varying Networks with Application to Wasserstein Barycenters}

\author[1]{\fnm{Olga} \sur{Yufereva}}\email{olga.o.yufereva@gmail.com}

\author[2]{\fnm{Michael} \sur{Persiianov}}\email{persiianov.mi@phystech.edu}

\author[3]{\fnm{Pavel} \sur{Dvurechensky}}\email{pavel.dvurechensky@wias-berlin.de}

\author[2,4,5]{\fnm{Alexander} \sur{Gasnikov}}\email{gasnikov@yandex.ru}

\author[6]{\fnm{Dmitry} \sur{Kovalev}}\email{dakovalev1@gmail.com}

\affil[1]{\orgname{N.\,N. Krasovski Institute of Mathematics and Mechanics}, \orgaddress{\city{Yekaterinburg}, \country{Russia}}}

\affil[2]{\orgname{Moscow Institute of Physics and Technology}, \orgaddress{\city{Dolgoprudny}, \country{Russia}}}

\affil[3]{\orgname{Weierstrass Institute for Applied Analysis and Stochastics}, \orgaddress{\city{Berlin}, \country{Germany}}}

\affil[4]{\orgname{Skoltech}, \orgaddress{\city{Moscow}, \country{Russia}}}

\affil[5]{\orgname{Institute of information transmission problems}, \orgaddress{\city{Moscow}, \country{Russia}}}

\affil[6]{\orgname{Universit\'e catholique de Louvain (UCL)}, \orgaddress{\city{Louvain-la-Neuve}, \country{Belgium}}}

\maketitle

\begin{abstract}

Inspired by recent advances in distributed algorithms for approximating Wasserstein barycenters, we propose a novel distributed algorithm for this problem. The main novelty is that we consider time-varying computational networks, which are motivated by examples when only a subset of sensors can make an observation at each time step, and yet, the goal is to average signals (e.g., satellite pictures of some area) by approximating their barycenter. We embed this problem into a class of non-smooth dual-friendly distributed optimization problems over time-varying networks, and develop a first-order method for this class. We prove non-asymptotic accelerated in the sense of Nesterov convergence rates and explicitly characterize their dependence on the parameters of the network and its dynamics. In the experiments, we demonstrate the efficiency of the proposed algorithm when applied to the Wasserstein barycenter problem.

\end{abstract}

\keywords{Distributed optimization \and  dual oracle \and  Wasserstein barycenter \and  time-varying networks \and  consensus problem
}

\section{Introduction}

Optimal transport (OT) \cite{monge1781memoire,kantorovich1942translocation} is getting more and more attention from the machine learning and optimization community motivated by a long list of applications such as unsupervised learning, semi-supervised learning, clustering, text classification, image retrieval, and others, see \cite{peyre2019computational} and references therein.
Given a basis space (e.g., pixel grid) and a transportation cost function (e.g., squared Euclidean distance), the OT approach defines a distance between two objects (e.g., images), modeled as two probability measures on the basis space, as the minimal cost of transportation of the first measure to the second.
Such distances, in particular, the Wasserstein distance, naturally capture the geometry of the data since they are invariant to shifts and rotations. In particular, the Frech\'e mean with respect to the Wasserstein distance, called Wasserstein barycenter (WB) \cite{agueh2011barycenters}, allows \cite{cuturi2014fast,barrio1999central} to reconstruct a template image from its random observations obtained by random shifts and rotations.

The benefits of the use of WB in real-world applications are sometimes outweighted by the large computational burden introduced by their definition. The computation of the Wasserstein distance is already a large-scale optimization problem, and to calculate the WB, one introduces a second optimization level as the WB minimizes the sum of Wasserstein distances to a set of probability measures. At this point, distributed optimization algorithms turned out to be efficient to scale-up the computations of the WB when the data is distributedly stored by a computational network \cite{staib2017parallel,uribe2018distributed,dvurechenskii2018decentralize,kroshnin2019complexity,dvinskikh2019primal,krawtschenko2020distributed,dvinskikh2021decentralized,rogozin2021decentralized}. On the other hand, the nature of the data-generating process may be distributed itself. In particular, it may be impossible to collect the data even in one datacenter, especially if the data processing has to respect the privacy of the individual data. Another example is a network of sensors that measure signals following some distributions and for the analysis purposes the whole network needs to find the WB of these distributions by peer-to-peer communications. In this case, decentralized %
algorithms are especially useful. Moreover, algorithms adapted to time-varying networks of sensors or computing devices are required. For example, some nodes of the network may be disconnected due to failures or, e.g., when the node is a satellite that observes certain area, the observation is available only for a certain period of time. At the same time, the development of decentralized distributed algorithms for general optimization problems is important on its own since the WB problem represents only one, yet important, example where such algorithms are efficient.

Summarizing, the WB problem is important for applications, yet requires to solve a large-scale optimization problem. For that problem, and other large-scale problems, decentralized distributed optimization algorithms on time-varying networks has to be developed. In this paper, we develop a general decentralized accelerated gradient method on time-varying networks and apply it to the WB problem.
\subsection{Related work}
In general, decentralized distributed optimization is an emerging and actively developed branch of optimization, see the recent survey \cite{gorbunov2022recent}.  
Our main focus in this paper is on the setting of decentralized methods working on time-varying networks, for which in the smooth and strongly convex case efficient algorithms were recently proposed in \cite{rogozin2021accelerated,li2021accelerated,kovalev2021lower} (primal oracle) and \cite{kovalev2021adom} (dual oracle). Unfortunately, these methods are not directly applicable in our case since the WB problem is not smooth, not strongly convex, and has simple constraints. At the same time, the WB problem has tractable dual oracle, which motivated us to extend the ADOM algorithm of \cite{kovalev2021adom} to the non-smooth, non-strongly convex setting with simple constraints.

Starting with \cite{uribe2018distributed,dvurechenskii2018decentralize} it was observed that decentralized methods with dual oracle are well suited for the WB problem. In the cycle of subsequent papers \cite{dvurechenskii2018decentralize,kroshnin2019complexity,dvinskikh2019primal,krawtschenko2020distributed,dvinskikh2021decentralized} different decentralized accelerated (randomized) algorithms were proposed for dual WB problem. In \cite{dvinskikh2020improved} the authors propose to reformulate the WB problem as a bilinear saddle-point problem~(SPP). Decentralized algorithm for this problem was proposed in \cite{rogozin2021decentralized}. Unfortunately, all these algorithms are designed for static networks that do not change over time. 

Moreover, their extensions to time-varying networks seem to be hardly possible. At the core of these algorithms lies the reformulation of the WB problem as a problem with linear constraints ensuring the consensus between the network nodes, and then solving the dual for that problem.
If communication network changes over time, then the affine-consensus constraints also change over time, and so does the dual problem. This essentially requires to solve a family of dual problems, which is not possible by the accelerated gradient methods or the Mirror-Prox algorithm as in \cite{dvurechenskii2018decentralize,kroshnin2019complexity,dvinskikh2019primal,krawtschenko2020distributed,dvinskikh2021decentralized,rogozin2021decentralized}. 

The recent work \cite{bishop2021network} considers the WB problem on time-varying networks and analyses a simple consensus method. The main difference with our paper is that they prove asymptotic convergence rather than convergence rates, and they consider possibly continuous measures on $\R$ unlike our setting of discrete measures on $\R^n$. The paper~\cite{wu2019fenchel} proposes Fenchel dual gradient methods for distributed convex optimization over time-varying networks. The main difference is that we propose an accelerated algorithm with better complexity, yet under a stronger assumption on the network (they assume the $B$-connectivity of the network). %

\subsection{Our contributions}
Since the WB problem has an efficient dual oracle, a natural idea is to use ADOM \cite{kovalev2021adom} that is an optimal decentralized algorithm for smooth strongly convex unconstrained problems for time-varying networks with dual oracle. 
ADOM can be considered as projected accelerated algorithm with inexact consensus-based projection applied to specific dual reformulation of the distributed optimization problem.
Since the WB problem
\begin{itemize}
    \item (Smoothness) is not smooth;
    \item (Constraints) is not unconstrained; as the space of probability measures  has simple constraints;
    \item (Strongly convex) is not strongly convex;
\end{itemize}
the direct application of ADOM to the WB problem is not possible. Moreover, not only WB problem, but also general non-smooth, non-strongly-convex constrained optimization problems lack efficient algorithms on time-varying networks.

The first main result of this paper is a generalization of ADOM for general $\gamma$ strongly convex decentralized optimization problems  with simple constraints on time-varying networks; solving it numerically requires $\cO\left(\frac{\lmax}{\lminp}\frac{1}{\sqrt{\gamma\varepsilon}}\ln \frac{1}{\gamma\varepsilon}\right)$ iterations. %
The second contribution is the application to the WB problem on time-varying networks; where the corresponding iteration number becomes $\cO\left(\frac{1}{\varepsilon}\ln \frac{1}{\varepsilon}\right)$. 
The main ideas are the following. To obtain the strong convexity we use the regularization of the primal problem by the entropy~\cite{peyre2019computational}. To deal with the constraints and non-smoothness, we use special regularization of the dual problem. This regularization goes back to \cite{devolder2012double,gasnikov2016efficient} and by infimal convolution can be considered as the Moreau--Yosida smoothing of the primal problem \cite{rockafellar1997convex,lemarechal1997practical}. We emphasize that the proposed dual regularization (primal smoothing) was earlier investigated only for non time-varying networks \cite{uribe2020dual}. For time-varying networks and problems with simple constraints the analysis is different.  

\paragraph*{Paper organization} Section \ref{sec: prelim} presents preliminaries and basic definitions for a general distributed optimization problem. Section \ref{sect: main results} states our main results, i.e. the method, convergence rates and the parameter estimation. Section \ref{sec: was} introduces the Wasserstein barycenter problem and specifies main results for the particular case. Section \ref{app: experiments} shows numerical experiments that illustrate and verify the theoretical results. All the proofs are in Appendices.

\paragraph*{Notation}
We utilize the following dimensions:
\begin{itemize}
    \item $m$ for the number of individual devices (nodes),
    \item $d$ for the dimension of a data in each device.
\end{itemize}
We use bold or normal font ($\bx$ or  $x$) for different spaces $\bx \in (\R^d)^m, \ x\in \R^d$. The $l$-th component of a vector $x\in \R^d$ is denoted by $[x]_l$ and $l$-th component of $\bx\in (\R^d)^m$ is denoted by  $[\bx]_l$ which is the corresponding vector from $\R^d$.

 Let $\boldOne$ denote a column vector with all entries equal to $1$. The $d$-dimensional simplex is denoted $S_1(d)$, that means $S_1(d) := \{p\in [0,1]^d \mid p\transpose \boldOne = 1\}$. For matrices $A$ and $B,$ $A \circ B$ and $A/B$ stands for the element-wise product and division, respectively. Another product we define as follows $\left\langle M,X \right\rangle  :=  \sum_{i=1}^{d} \sum_{j=1}^{d} M_{ij} X_{ij}$. %
 
 Abbreviation WB means Wasserstein barycenter and ADOM refers to Accelerated Decentralized Optimization Method proposed in \cite{kovalev2021adom}.

\section{Decentralized optimization}
\label{sec: prelim}

\subsection{Decentralized computation problem}
Decentralized computation simulates computation on distributed individual devices. The devices are considered as nodes of an undirected connected graph called a {\em communication network}. It means that each node can perform computations based only on its local data and the data of its neighbors in communication network. For a convex closed set $\set$ and convex functions $f_i$
 {\em decentralized computation of the following optimization problem}
\begin{equation}
\label{eq-generic_decentralized_op}
	 \min\limits_{x\in \set}\sum \limits_{i=1}^m f_i(x)
\end{equation}
requires numerical computation assuming that each function $f_i$ is stored on the corresponding node $i\in [m]\eqdef \{1,2,\ldots,m\}$. Such approach brings us to an effective reformulation of the optimization problem.

\subsubsection{Consensus condition}

Since each computational node carries its own local data approximation, we can substitute formally different variables $x_i$ for the mutual argument $x$ in \eqref{eq-generic_decentralized_op} assuming that they belong to the so-called {\it consensus space}. 
It means the new variables $x_i$ must eventually coincide with each other and with the wanted barycenter. We obtain an equivalent optimization problem in the following form:
\begin{gather}
 \min\limits_{\mathbf{x}\in \cS}F(\mathbf{x}) =  \min\limits_{\mathbf{x}\in \cS}\sum\limits_{i=1}^{m} f_{i}([\bx]_i), \label{eq-generic_consensus_op}\\
\text{where } \cS = \left\{\bx=([\bx]_1,\ldots,[\bx]_m)\in (S)^m \mid [\bx]_1 = \ldots=[\bx]_m\right\}. \notag
\end{gather}
Here $i$-th component $[\bx]_i$ is a corresponding $d$-dimensional vector. 

\subsubsection{Time-varying communication network}

We consider $m$ distributed devices that seek to reach a consensus solution of an optimization problem. The devices are connected via an $m$-node network that changes over time. At each time step $n$ we denote Laplacian of the corresponding network by $\hat W_n$. In general, it suffices to take any $\hat W_n$ satisfying the following:
\begin{enumerate}
	\item $\hat{W}_n$ is symmetric and positive semi-definite,
	\item $[\hat{W}_n]_{i,j} \neq 0$ if and only if $(i,j)$ are connected by the network,
	\item $\ker \hat{W}_n = \{(x_1,\ldots,x_m) \mid x_1 = \ldots = x_m\}$.
\end{enumerate}
Further we use {\em communication matrix} that is the block matrix $\bW_n = \hat{W}_n \otimes I_d$. %
Hence, decentralized communication of each vector $x_i$ stored on the $i$-th node at a time step~$n$ can be represented by  multiplication of the $md$-dimensional vector $(x_1, \ldots, x_m)$ and matrix $\bW_n$: indeed,  if  $\by = \bW_n \bx$, then it yields 
\begin{equation*}
[\by]_i = \sum \limits_{j=1}^m [\hat W_n]_{i,j} [\bx]_j = \sum \limits_{j \in \cN_i} [{\hat W}_n]_{i,j} [\bx]_j,
\end{equation*}
where $\cN_i$ is the set of the neighboring nodes for the node $i$ according to the communication network at $n$-th iteration.
Thus, for each node $i$, vector $[\by]_i$ is a linear combination of vectors $[\bx]_j$, stored at the neighboring nodes $j \in \cN_i$. %

The considered algorithms require these communication matrices $\bW_n$ to have conditional numbers bounded for all $n \in \{0,1,2\ldots\}$. Namely, we utilize the following assumption.
\begin{assumption}
\label{assumption-lambdas}
Let there exist constants $0 < \lminp < \lmax$ such that
\begin{equation*}%
	\lminp \leq \lminp({\bW}_n) \leq \lmax({\bW}_n) \leq \lmax \qquad \forall n,
\end{equation*}
where $\lminp({\bW}_n)$ is the smallest positive eigenvalue of $\bW_n$ and $\lmax({\bW}_n)$ is the biggest one.
\end{assumption}
A condition number of the matrix ${\bW}_n$ is given as $\frac{\lmax({\bW}_n)}{\lminp({\bW}_n)}$ and relates to the  connectivity of a network;  it appears in convergence rates of many decentralized algorithms. %

\section{Main results}
\label{sect: main results}

One of our main features is that general convex functions of optimization problem can be defined on a convex set $S\subseteq \R^d$ instead of the entire $\R^d$, that was crucial for duality of \eqref{eq: opt prob for ADOM setup with bx^*} and \eqref{eq-prob-H*}. %
First of all, we study the case of strongly convex functions (Theorem~\ref{theorem: strongly convex functions}).
Then we see that similarly we can approximate convex functions (Corollary~\ref{cor: convex functions}). We obtain an important Theorem~\ref{th: wasserstein} by applying Theorem~\ref{theorem: strongly convex functions} to the Wasserstein barycenter problem; in this particular setup one can estimate parameters more precisely. In the next sections we introduce necessary definitions, state Theorem~\ref{th: wasserstein}, and provide related numerical experiments. All the proofs are in the Appendices.

Consider $\gamma$ strongly convex case, i.e. the following decentralized optimization problem %
\begin{gather}
    \bx_{\gamma}^* =  \argmin\limits_{\bx\in \cS} F^{\gamma}(\bx) = \argmin\limits_{\bx\in \cS} \sum\limits_{i=1}^{m} f_i^{\gamma}([\bx]_i)= \argmin\limits_{x\in S} \sum\limits_{i=1}^{m} f_i^{\gamma}(x), \label{eq: op-f gamma}\\
\mbox{where } \cS = \{\bx=([\bx]_1,\ldots,[\bx]_m)\in (S)^m \subset (\R^d)^m \mid [\bx]_1 = \ldots=[\bx]_m\}, \notag
\end{gather}
functions $f_i^{\gamma}$ are  $\gamma$ strongly convex, differentiable, and defined on a convex set~$S$. Recall that $i$-th component $[\bx]_i$ is a corresponding $d$-dimensional vector. 
\begin{theorem}
\label{theorem: strongly convex functions}
Let $S\subset \R^d $ be a convex set, let functions $f_i^{\gamma}\colon S \to \R$ of the problem~\eqref{eq: op-f gamma} be $\gamma$ strongly convex and  differentiable,
let $\bW_n$ be the $n$-th communication matrix satisfying Assumption~\ref{assumption-lambdas} for some $\lminp, \lmax>0$. 
For any $r>0$, after $n$ iterations of Algorithm \ref{alg:mod-adom} we obtain  $\bx^{n}_{r,\gamma} =  \nabla H^*(\bz^n_{g})$ that provides:
\begin{enumerate}
        \item consensus condition approximation: for each $i$ and $j$
        \begin{equation}
            \label{eq: theorem: consensus}
            \left\|\left[\bx^{n}_{r,\gamma}\right]_i - \left[\bx^{n}_{r,\gamma}\right]_j\right\|^2_2\leq C_1 \left(1- \frac{\lminp}{7\lmax} \sqrt{\frac{r\gamma}{1+r\gamma}}\right)^n;
        \end{equation}
        \item value approximation:
        \begin{equation}
        \label{eq: theorem: value approximation}
              F^{\gamma}(\bx^{n}_{r,\gamma})- \min\limits_{\bx\in \cS}F^{\gamma}(\bx)\leq \frac{r}{2(1+r\gamma)}mK^2 + C_2 \left(1- \frac{\lminp}{7\lmax} \sqrt{\frac{r\gamma}{1+r\gamma}}\right)^{n/2};
        \end{equation}
    \end{enumerate}
where $K$ is such that $\|\nabla f^{\gamma}_i(x)\|_2< K$ for each $i$ and for all $x$ from $\varepsilon/\gamma$-neighborhood of the solution $\argmin\limits_{x\in S}\sum\limits_{i=1}^{m}f^{\gamma}_i(x)$.  %
The parameters are $C_1=\frac{(1+r\gamma)^2}{2\gamma^2}$, $C_2 = %
\frac{m(1+r\gamma)K}{\sqrt{2}\gamma}\sqrt{\frac{\lmax}{\lminp}} 
    + \frac{m(1+r\gamma)^2}{4r\gamma^2}$.
\end{theorem}
\begin{proof}
See Appendix \ref{proof main theorem}.
\end{proof}
\begin{remark}
\label{rem: numb of iterations}
To reach $\varepsilon$ approximation of \eqref{eq: theorem: consensus} and \eqref{eq: theorem: value approximation} it suffices to take $r\leq \frac{\varepsilon}{2mK^2}$.
 Then the rate of the number of iterations is 
 \[n= \cO\left(\frac{\lmax}{\lminp}\sqrt{\frac{1+r\gamma}{r\gamma}}\ln \frac{C_2}{\varepsilon}\right) = \cO\left(\frac{\lmax}{\lminp}\frac{1}{\sqrt{\gamma\varepsilon}}\ln \frac{1}{\varepsilon}\right).\]
\end{remark}
\begin{proof}
See Appendix \ref{proof numb of iterations}.
\end{proof}
\begin{algorithm}[H]
	\caption{Modified ADOM}
	\label{alg:mod-adom}
	\begin{algorithmic}[1]
		\State {\bf input:} $r>0$, for $i = 1,\ldots,m$: $f^{\gamma}_i\colon S \to \R$, $(f_i^{\gamma})^*(z) = \sup\{\langle z,x\rangle - f_i^{\gamma}(x)\mid x\in S\}$ 
  \State define $\nabla h_i^*([\bz]_i) =  \nabla (f_i^{\gamma})^*([\bz]_i) +  r[\bz]_i$ \label{eq: theorem declaration nablah^*}
		\State define $\nabla H^* (\bz) = \left(\nabla h_1^*([\bz]_1),\ldots,\nabla h_m^*([\bz]_m)\right)\transpose$ \label{eq: theorem declaration nablaH^*}
		\State set
		$\alpha = \frac{r}{2}$, $\eta = \frac{2\lminp\sqrt{\gamma}}{7\lmax\sqrt{r(1+r\gamma)}}$, $\theta = \frac{\gamma}{\lmax(1+r\gamma)}$, $\sigma = \frac{1}{\lmax}$, $\tau = \frac{\lminp}{7\lmax}\sqrt{\frac{r\gamma}{1+r\gamma}}$
		\State set $\bz^0 =\mathbf{0}$, $\bz_f^0 = \bz^0$, $\boldm^0 = \mathbf{0}$
		\For{$n = 0,1,2,\ldots$}
		\State $\bz_g^n = \tau \bz^n + (1-\tau)\bz_f^n$
		\State $\Delta^n = \sigma\bW_n(\boldm^n - \eta\nabla H^*(\bz_g^n))$
		\State $\boldm^{n+1} = \boldm^n - \eta\nabla H^*(\bz_g^n) - \Delta^n$
		\State $\bz^{n+1} = \bz^n + \eta\alpha(\bz_g^n - \bz^n) + \Delta^n$
		\State $\bz_f^{n+1} = \bz_g^n - \theta\bW_n\nabla H^*(\bz_g^n)$
		\EndFor
		\State {\bf output: $\bx^{n}_{r,\gamma} =  \nabla H^*(\bz^n_{g})$}
	\end{algorithmic}
\end{algorithm}
\begin{corollary}
\label{cor: convex functions}
Let $S$ be  a convex set in $\R^d$,  let $f_i\colon S\to \R$ be differentiable convex functions for $i=1,\ldots,m$, and let $\bW_n$ be the $n$-th communication matrix satisfying Assumption~\ref{assumption-lambdas} for some $\lminp, \lmax>0$. Decentralized convex optimization problem $\min_{x\in S}\sum_{i=1}^{m} f_i(x)$
over time-varying communication networks  
can be $\varepsilon$ approximated numerically by Algorithm \ref{alg:mod-adom} %
applied 
for $\gamma$ strongly convex regularizing functions\footnote{e.g., one can take $f_i^{\gamma}(x) = f_i(x) + \frac{\gamma}{2}\|x\|^2_2$} $f_i^{\gamma}(x)$ that satisfy  
\begin{equation}
    \label{cor: stron convex}
    0\leq \min\limits_{x\in S} \sum\limits_{i=1}^m f_i(x) - \min\limits_{x\in S}\sum\limits_{i=1}^m f_i^{\gamma}(x)  \leq \varepsilon/2
\end{equation}
if $r<\frac{\varepsilon}{4mK^2}$ where $K$ is such that $\|\nabla f^{\gamma}_i(x)\|_2< K$ for each $i$ and for all $x$ from $\varepsilon/\gamma$-neighborhood of the solution $\argmin\limits_{x\in S}\sum_{i=1}^{m}f^{\gamma}_i(x)$. 
Moreover, if $\gamma = \sqrt{\varepsilon}$, then the rate of the number of iterations is 
 \[n=  \cO\left(\frac{\lmax}{\lminp}\frac{1}{\varepsilon}\ln \frac{1}{\varepsilon}\right).\]
\end{corollary}
\begin{proof}
The condition $r<\frac{\varepsilon}{4mK^2}$ follows that the right-hand side of the inequality \eqref{eq: theorem: value approximation} is less or equal $\varepsilon/2$, i.e. $$\sum_{i=1}^m f_i^{\gamma}([\bx^n_{r,\gamma}]_i) - \min_{x\in S} \sum_{i=1}^m f_i^{\gamma}(x)\leq \varepsilon/2.$$ Combining it with \eqref{cor: stron convex} we obtain $\varepsilon$ approximation of the convex constrained optimization problem $\min\limits_{x\in S}\sum\limits_{i=1}^{m} f_i(x)$.
\end{proof}

\section{Wasserstein Barycenter Problem}
\label{sec: was}

Wasserstein barycenter problem is a motivating and challenging convex constrained optimization problem  that is convex, but not smooth and not strongly-convex. The problem belongs to the optimal transport theory which %
recently got various application in, e.g., texture mixing \cite{rabin2011wasserstein}, statistical estimation of template models \cite{boissard2015distribution}, graphics and machine learning (for regression, classification and generative modeling) \cite{peyre2019computational}. %
Further we consider WB problem as a decentralized convex optimization problem and propose a computation method in Theorem \ref{th: wasserstein}.

\subsection{Wasserstein distance}

 We provide here only necessary definitions and take into consideration only finite-supported distributions since we deal with numerical experiments. General theory, that begins with Wasserstein distance, can be found in \cite{villani2009optimal}.

Recall that we denote $d$-dimensional simplex by $S_1(d)$ and it represents a set of possible probabilities distributions as $S_1(d) = \{p\in [0,1]^d \mid \sum_{i=1}^d [p]_i = 1\}$. Consider two probability distributions $p, q \in \simplex$ with support on a finite set of points $\{\omega_i \in \mathbb{R}^n \}_{i=1}^d$ such that $p(\omega_i) = p_i$ and $q(\omega_i) = q_i$. %
Then, a cost (loss) matrix $M$ is such that its element $[M]_{ij}\in \mathbb{R}_+$
 represents the cost of moving a unit mass from $\omega_i$ to $\omega_j$. So $M$ is a non-negative symmetric matrix with zeros on the diagonal.  It is often taken as the Euclidean distances matrix, i.e. $[M]_{ij} = \|\omega_i - \omega_j\|_2^2$. %
The set of {\it transport plans} is defined as
		\begin{align*}
		U(p,q)  := \left\lbrace X \in \mathbb{R}_+^{d \times d} \;\big|\; X \boldOne = p, X^T\boldOne = q \right\rbrace,
		\end{align*}
i.e. the set of probabilities measures on $\mathbb{R}_+^{d \times d}$ with margins $p$ and $q$. 
{\it Wasserstein distance} between two probability distributions defines as the following minimum among component-wise multiplication of the cost matrix and transport plans:
\[\cW (p,q)  := \min_{X \in U(p,q)}  \langle M, X \rangle. \]

\subsection{Wasserstein barycenter }

Wasserstein barycenter of a set of probability distributions $q_1,\ldots,q_m$ is a probability distribution itself that is defined as the solution to the following optimization problem %
\begin{equation}
		\label{w_barycenter}
			 \min_{p \in \simplex}\sum\limits_{i=1}^{m} \cW_{q_i}(p),
\end{equation}
where $\cW_{q_i}(p):=\cW(q_i,p)$.  %
In distributed approach, each device (node) possesses its original distribution $q_i$ and the corresponding function $f_i(\cdot) = \cW_{q_i}(\cdot)$. The goal of the whole system is, by communicating with each other, to approximate the barycenter by an iterative algorithm.  At each iteration each node computes a new guess for the barycenter distribution using current guesses from its neighbors. Typically, the first guess coincide with the original distribution $q_i$ and the resulting distributions reach a consensus.  It is known (see e.g. \cite{cuturi2014fast}) that the WB captures the mean structure of given data.
On example of a dataset of hand-written digits `4' of MNIST 784 \cite{lecun1998mnist} Figure \ref{fig: visual digits} shows how local nodes' guesses change and tend to global barycenter, which resembles a digit `4' as well. That visually illustrates our theoretical result, Theorem \ref{th: wasserstein}. %
In this experiment communication networks are \erdos \  random networks and change every 5 iterations.
\begin{figure}[h!t]
\centering
\subfloat[initial data samples]{
\resizebox*{\textwidth}{!}{\includegraphics[width=\textwidth]{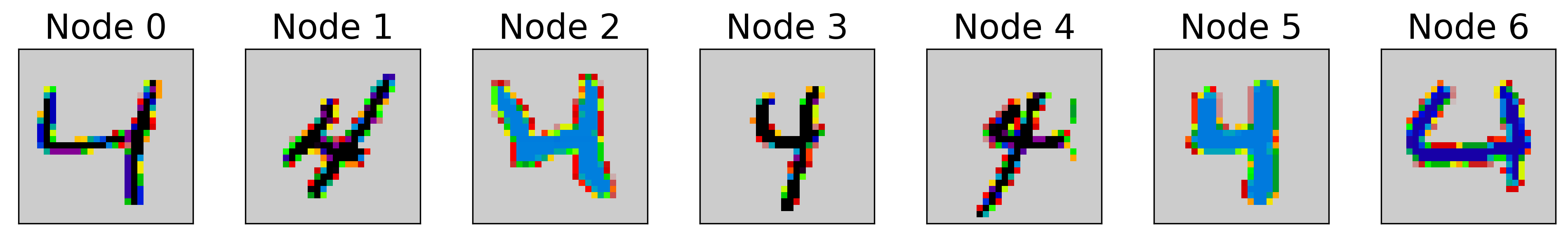}}}
\hfill
\subfloat[iteration 10]{
\resizebox*{\textwidth}{!}{\includegraphics[width=\textwidth]{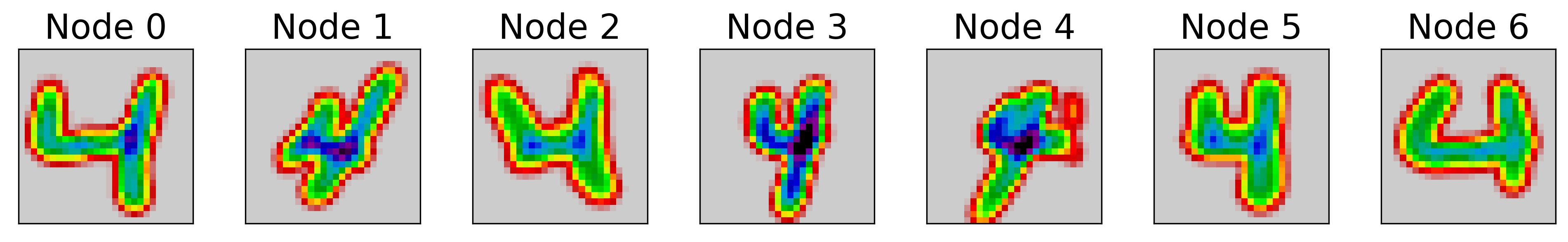}}}
\hfill
\subfloat[iteration 50]{
\resizebox*{\textwidth}{!}{\includegraphics[width=\textwidth]{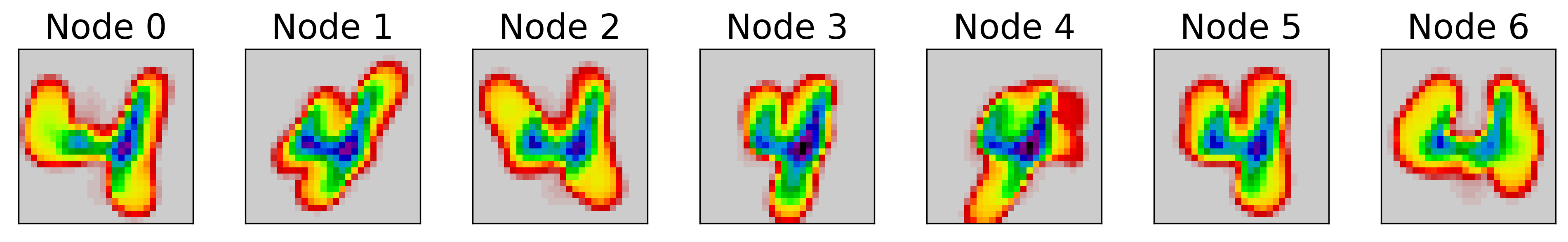}}}
\hfill
\subfloat[iteration 100]{
\resizebox*{\textwidth}{!}{\includegraphics[width=\textwidth]{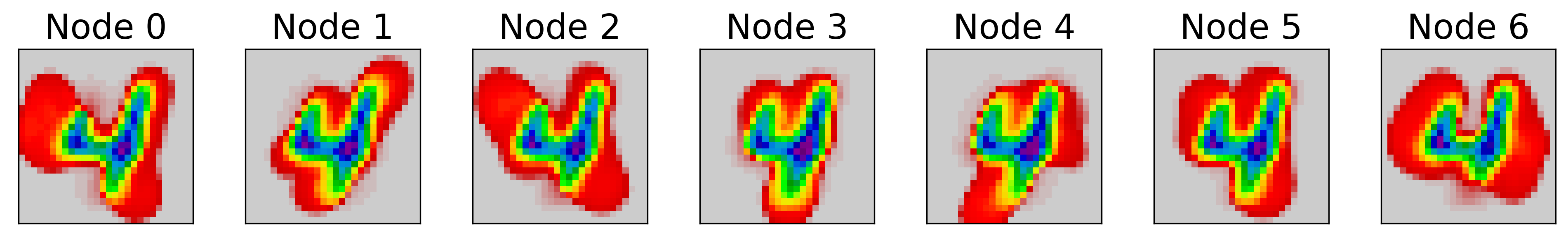}}}
\hfill
\subfloat[iteration 200]{
\resizebox*{\textwidth}{!}{\includegraphics[width=\textwidth]{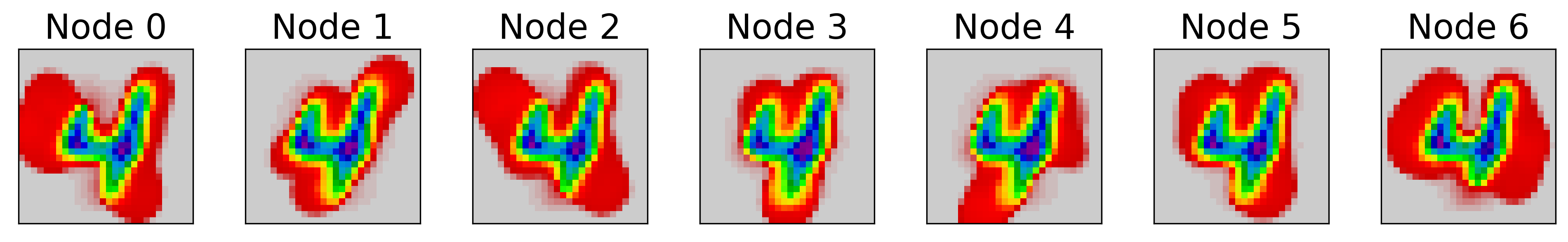}}}
\caption{Evolution of local data converging to Wasserstein barycenter of the hand-written digit 4 of the MNIST784 dataset for a subset of 7 nodes out of 50 over \erdos\ random networks varying each 5 iterations; regularization parameters are $\gamma=0.03, r=0.001$}
\label{fig: visual digits}
\end{figure}

\subsection{Resulting algorithm for WB problem}

To apply the proposed numerical scheme we use entropy regularization of Wasserstein distance $\cW$ and make the following assumption on the initial data. 
\begin{assumption}
\label{assumption}
Let vectors $q_i\in\simplex\subset\R^d$   be such  that \[\min\limits_{i=1,\ldots,m 
\atop l=1,\ldots,d}[q_i]_l = \delta >0.\] 
\end{assumption}
It is not too restrictive as it only excludes zero probabilities of states, which can be done by a little distortion. 
The entropy regularized Wasserstein distance is defined as 
	\begin{align}
		\cW_{\gamma, q} (p)  := \cW_{\gamma} (q,p)=  \min_{X \in U(p,q)} \left\lbrace  \langle M, X \rangle + \gamma \sum_{i=1}^{d} \sum_{j=1}^{d} X_{ij}\ln X_{ij} \right\rbrace,
		\end{align}
where $x\ln x$ is assumed to equal zero if $x=0$.

\begin{theorem}
\label{th: wasserstein}
Let initial distributions $q_i$ satisfy Assumption~\ref{assumption} and let $p^*$ be their Wasserstein barycenter, i.e.  $p^*$ minimizes the problem \eqref{w_barycenter}. 
Let communication matrices $\bW_n$ satisfy Assumption \ref{assumption-lambdas} for some $\lminp, \lmax>0$. If Algorithm \ref{alg:mod-adom} is applied for entropy regularized Wasserstein distance functions $f_i^{\gamma} = \cW_{\gamma, q_i}$, then functions $\nabla h_i^*(z)$ are defined as 
\begin{equation}
    \label{eq: th: wasserstein}
    \begin{array}{cc}
        \nabla h_i^*(z) =  \frac{r}{2}z+ \sum\limits_{j=1}^{m} [q]_j \frac{\exp(\frac{1}{\gamma}([z]_l - M_{lj}))}{\sum\limits_{i=1}^{m}\exp(\frac{1}{\gamma}([z]_i - M_{ij}))},
        \\
        \nabla (H^{r,\gamma})^* (\bz) = (\nabla h_1^*([\bz]_1),\ldots,\nabla h_m^*([\bz]_m))\transpose,
    \end{array}
\end{equation}
and it suffices to take $\gamma = \frac{1}{8}\varepsilon \ln d$, 
 $K^2 =  \sum\limits_{j=1}^{d}\left( 2\gamma\ln d + \inf_i\sup_l |M_{jl} - M_{il}| - \gamma\ln \frac{\delta}{2}\} \right)^2$, and
 $r= \frac{\varepsilon}{4mK^2}$
  to $\varepsilon$ approximate the solution $p^*$ as follows
\begin{gather*}
    \left|\sum\limits_{i=1}^m \cW_{q_i}([\bx^{n}_{r,\gamma}]_i) - \sum\limits_{i=1}^m \cW_{q_i}(p^*)\right|
    \\
    \leq
    2\gamma \ln d + \frac{r}{4(1+r\gamma)}mK^2 + C \left(1- \frac{\lminp}{7\lmax} \sqrt{\frac{r\gamma}{1+r\gamma}}\right)^{n/2}\leq \varepsilon.
\end{gather*}
 Thus, a sufficient number of iterations of Algorithm \ref{alg:mod-adom} is
 \[n= \cO\left(\frac{\lmax}{\lminp}\sqrt{\frac{1+r\gamma}{r\gamma}}\ln \frac{C}{\varepsilon}\right) = \cO\left(\frac{\lmax}{\lminp}\frac{1}{\varepsilon}\ln \frac{1}{\varepsilon}\right).\]
\end{theorem}

\section{Numerical Experiments}
\label{app: experiments}

We provide numerical experiments to demonstrate performance of the proposed method. %
We can fruitfully test our method on WB problem with an artificial set of univariate, discrete and truncated Gaussian distributions. For such a dataset, the resulting distribution (the zero-entropy Wasserstein barycenter) is described by an analytic formula.
Namely,  the barycenter is a Gaussian distribution which mean is the arithmetic average of the means of the given Gaussians and the standard deviation of the barycenter is the arithmetic average of the standard deviations of the given Gaussians.

For all figures of this section we generated a dataset of truncated Gaussians. Each distribution size is $100$, while a size of a dataset (i.e. the number of nodes) differs and is indicated at each figure. For entropy regularized Wasserstein distance we use normalized Euclidean cost matrix and entropy regularization parameter $\gamma = 0.01$. The regularization parameter $r$ of the method is $r = 0.001$.

\subsection{Comparison with other methods}

To the best of our knowledge, we can compare our method (called here ADOM) with local barycenters method (LB) proposed in  \cite{bishop2021network} and Fenchel dual gradient method (FDGM) proposed in \cite{wu2019fenchel}.  They all are applicable for the WB problem on time-varying networks. %
Regardless of the analytical form of the LB algorithm, in order to implement it we need either to solve an optimization problem at each iteration or to use approximations, e.g. methods of \cite{flamary2021pot}. Note that, for any particular setup, realization of FDGM is quite a problem, since the method is sensitive to the step size $\alpha_n$. ADOM negotiates limitations described above and reveals relatively stable convergence as it is presented at Figures~2--3%
, where we test the methods on cycle networks that change every iteration.   %

\begin{figure}[H]
\label{fig: comparison_opt}
\centering
\includegraphics[width=\textwidth]{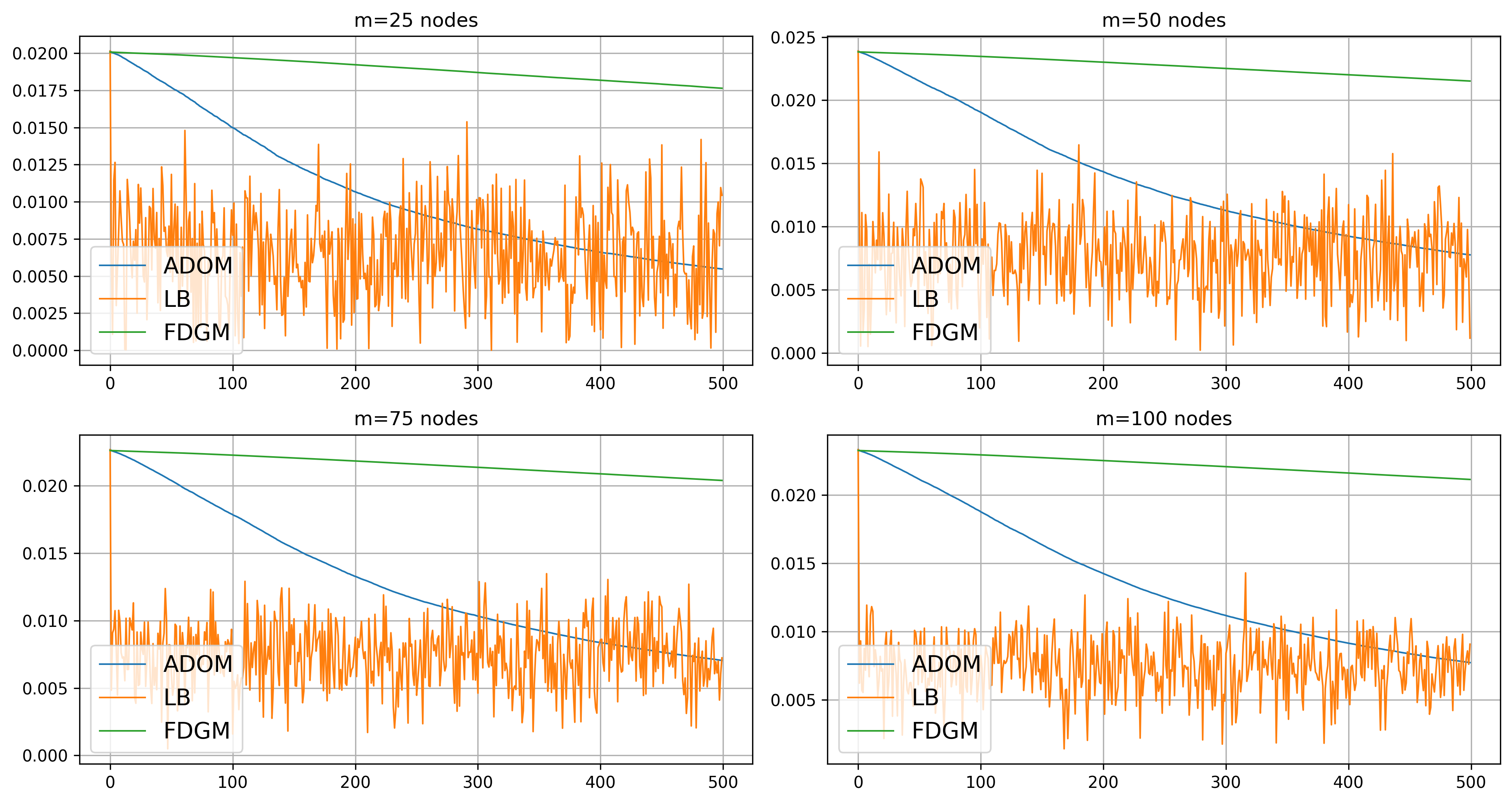}
\caption{$\frac{1}{m}\left(\sum\limits_{i=1}^m\cW(q_i, [\bx^n\rgam]_i) - \sum\limits_{i=1}^m\cW(q_i, p^*)\right)$-convergence comparison on cycle networks changing every iteration}
\end{figure}

\begin{figure}[H]
\centering
\includegraphics[width=\textwidth]{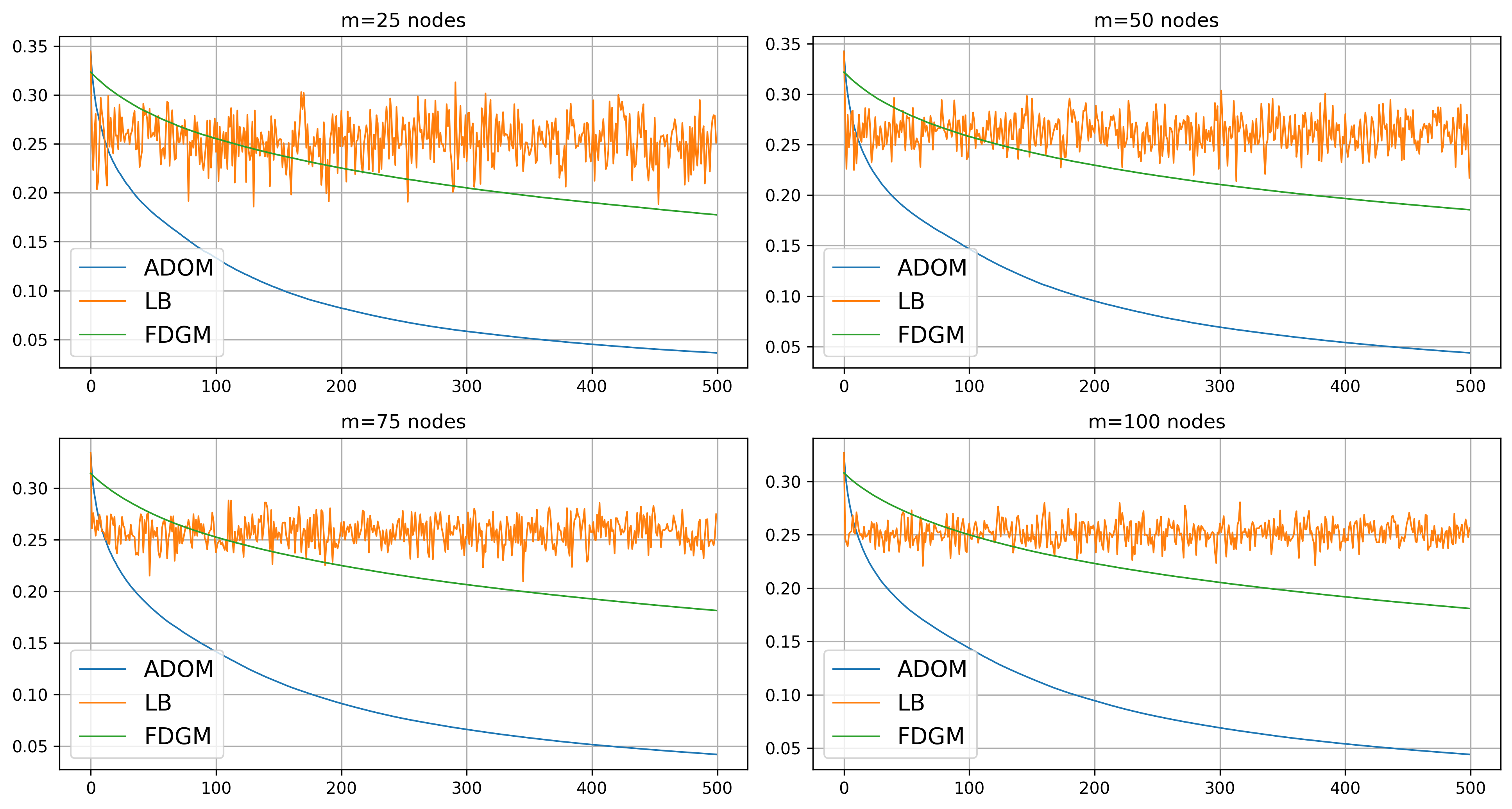}
\label{fig: comparison_con}
\caption{Consensus condition comparison of methods on cycle networks changing every iteration}
\end{figure}

\subsection{Performance rate}

We can see the difference in error value, i.e. in $\frac{1}{m}\left(\sum\limits_{i=1}^m\cW(q_i, [\bx^n\rgam]_i) - \sum\limits_{i=1}^m\cW(q_i, p^*)\right)$, in two opposite cases: Figure~\ref{fig: dif_topologies_constant} presents evolution of errors computed for constant networks of different topologies, while at Figure~\ref{fig: dif_topologies_varying} networks change every iterations within indicated topology; the exception is complete network that cannot change. One natural way to sample a random network is to independently sample each edge with a probability~$p$. Such networks are called \erdos\ network or $(m,p)$-\erdos\ network, where $m$ is the number of nodes and $p$ is the probability of an edge. Let us notice also that star, cycle and minimum spanning tree (of $(m, 0.9)$-\erdos) networks have $m-1, m,$ and $m+1$ edges respectively in contrast to the complete network with $m(m-1)/2$ edges and $(m, 0.5)$-\erdos\ network that has $m(m-1)/4$ edges in average. 

\begin{figure}[H]
\centering
\subfloat[Constant networks]{
\resizebox*{0.46\textwidth}{!}{\includegraphics[width=\textwidth]{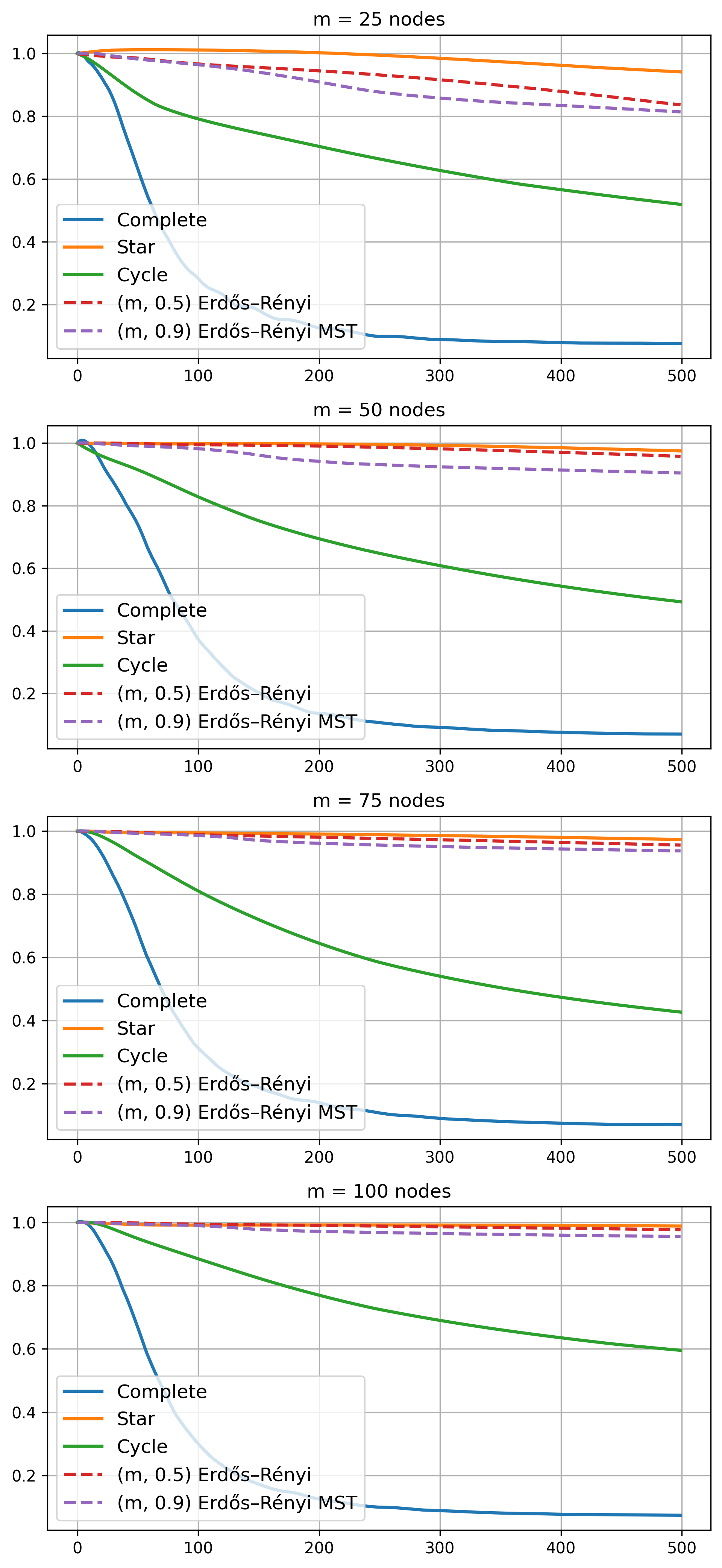}}
\label{fig: dif_topologies_constant}
}
\hfill
\subfloat[Time-varying networks]{
\resizebox*{0.46\textwidth}{!}{\includegraphics[width=\textwidth]{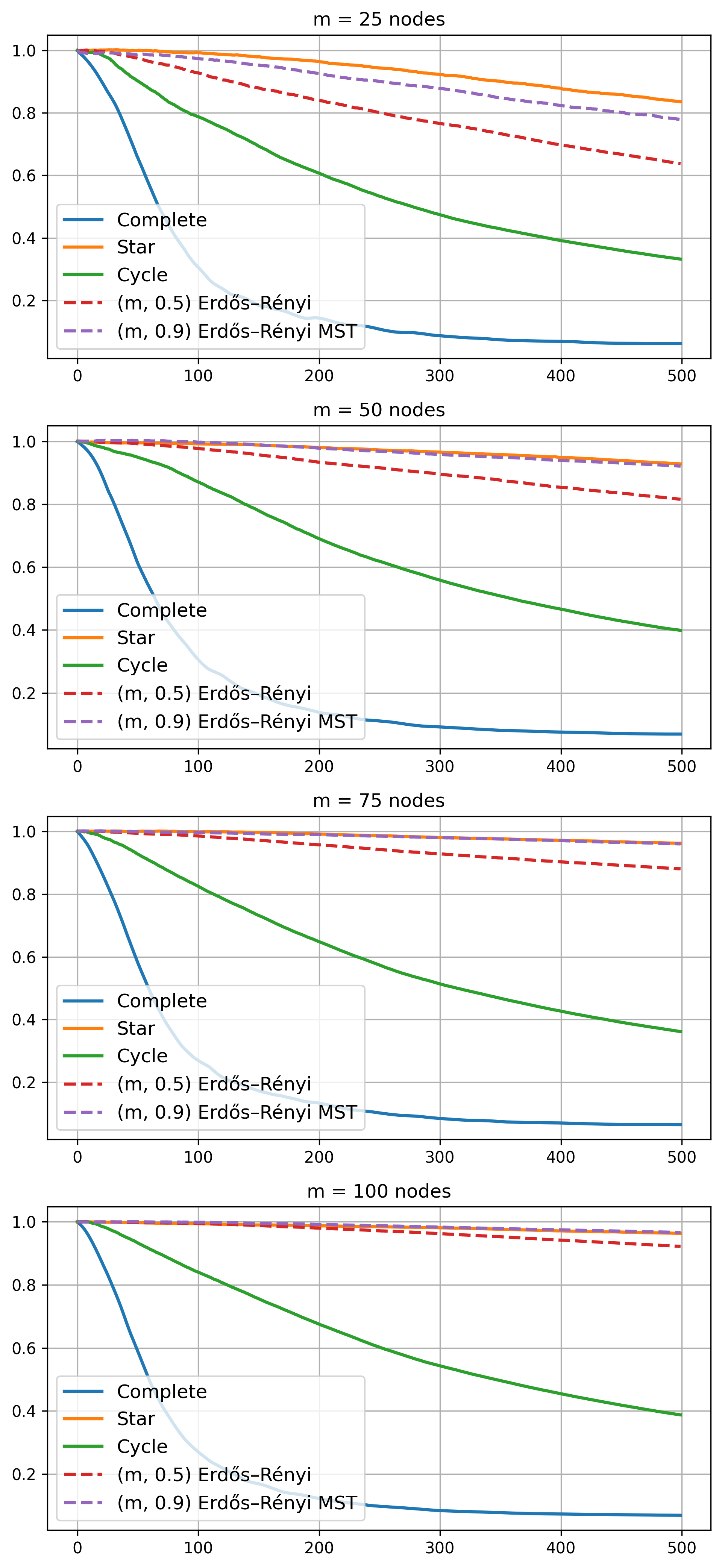}
\label{fig: dif_topologies_varying}}}
\label{fig: dif_topologies}
\caption{Different network topologies: error over time}
\end{figure}

For the two `most efficient' topologies that prove themselves at Figures~\ref{fig: dif_topologies_constant}--\ref{fig: dif_topologies_varying} we compute at Figure~\ref{fig: freq} the error evolution for different frequency of the networks varying. We indicate the lengths of epoch, i.e. number of iteration between network changing. Notice that the evolution for constant networks, computed above, matches to infinite  epoch length. The number of iterations remain the same on all figures despite it is insufficient for convergence on cycle networks. Nonetheless one can see the trends of convergence and notice that there is no monotonicity with respect to frequency of networks varying.

\begin{figure}[H]
\centering
\subfloat[Cycle  networks]{
\resizebox*{0.98\textwidth}{!}{\includegraphics[width=\textwidth]{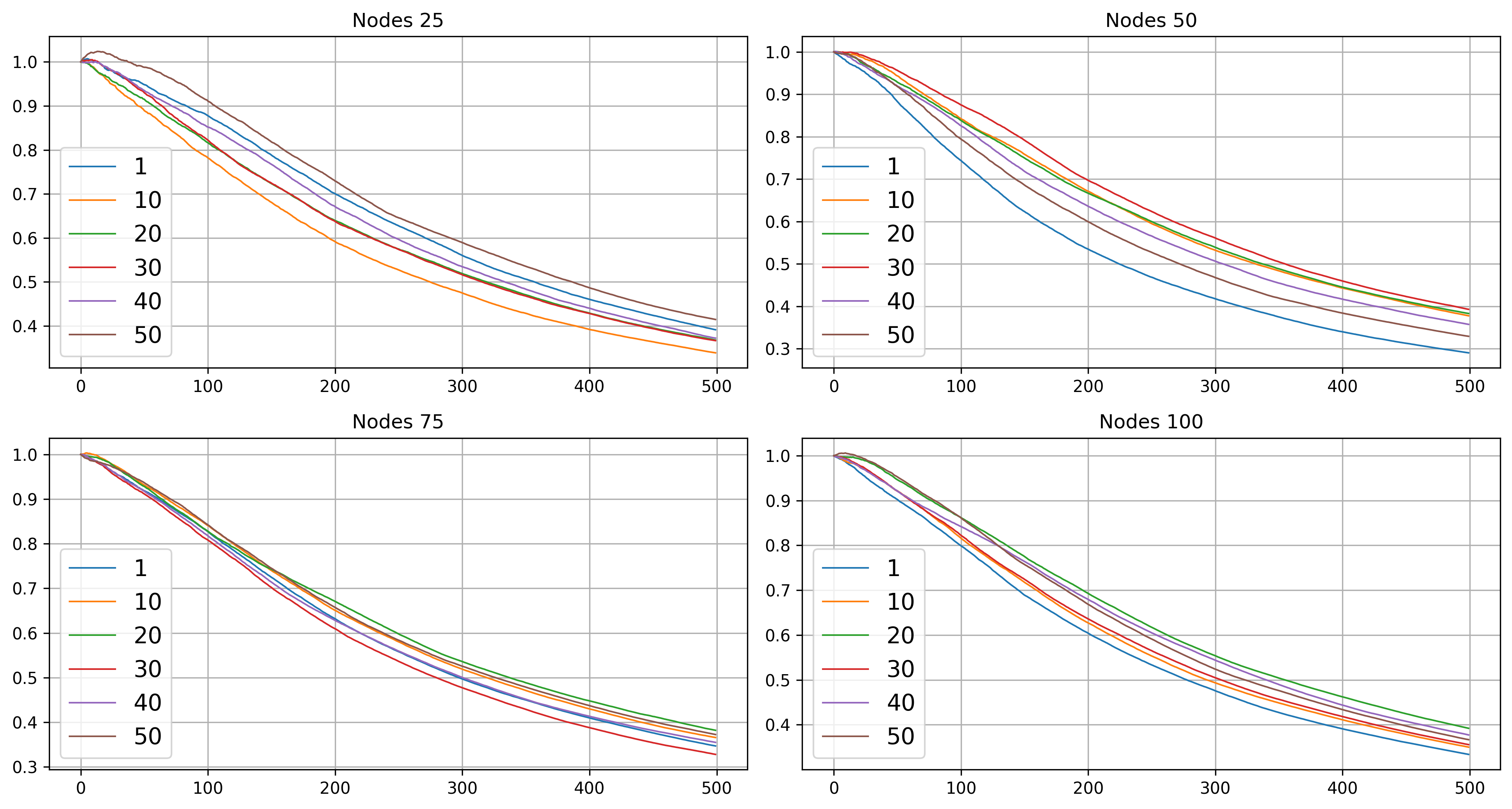}}
}%
\vfill
\subfloat[$(m, 0.9)$-\erdos\  networks]{
\resizebox*{0.98\textwidth}{!}{\includegraphics[width=\textwidth]{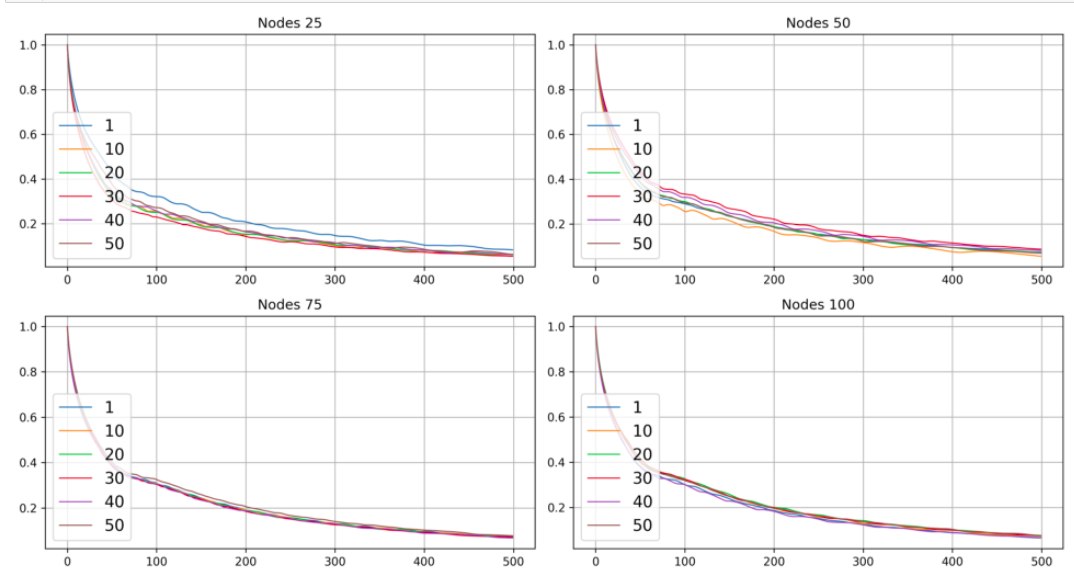}}
}
\caption{Different number of iteration between network changing:  errors over time}
\label{fig: freq}
\end{figure}

\section*{Acknowledgments} 
The authors are grateful to Alexander Rogozin.

The work of A. Gasnikov in Section 4 of the paper was funded by Russian Science Foundation (project 18-71-10108).  

The work of O. Yufereva in the rest part of the paper was performed as part of research conducted
in the Ural Mathematical Center with the financial support
of the Ministry of Science and Higher Education of the Russian
Federation
(Agreement number 075-02-2023-913).

\bibliography{sn-bibliography}

\appendix

\section{ADOM and its assumptions}

The state of the art numerical computation method for time-varying networks, called ADOM, is developed in \cite{kovalev2021adom} and this subsection is to present its main objects. It has natural restrictions on the class of suitable problems and, e.g., Wasserstein barycenter problem lies beyond the requirements of this algorithm. %
So we modify ADOM to solve more general optimization problems with restrictions.
For the sake of consistency, we slightly change original notation and adduce below the results from \cite{kovalev2021adom}.

In \cite{kovalev2021adom}, optimization problem with the consensus condition is
\begin{gather}
 \min\limits_{\mathbf{x}\in \cR} H(\mathbf{x}) =  \min\limits_{\mathbf{x}\in \cR} \sum\limits_{i=1}^{m} h_{i}([\bx]_i), \label{eq: opt prob for ADOM setup with bx^*}\\
\text{where } \cR = \left\{\bx=([\bx]_1,\ldots,[\bx]_m)\in (\R^d)^m \mid [\bx]_1 = \ldots=[\bx]_m\right\} \notag, 
\end{gather}
where functions $h_i\colon \R^d \to \R$ are assumed to be smooth and strongly convex.  Problem~\eqref{eq: opt prob for ADOM setup with bx^*} is equivalent to the following:
\begin{gather}
    \min\limits_{\mathbf{z}\in \cRperp}H^*(\mathbf{z}) \label{eq-prob-H*}, \\
    \text{where }
    \cRperp = \left\{\bz=([\bz]_1,\ldots,[\bz]_m)\in (\R^d)^m \;\Big|\; \sum\limits_{i=1}^m [\bz]_i = 0\right\}, \notag
\end{gather}
where $H^*$ is the Fenchel transform  of the function $H$ and $\cRperp$ is the orthogonal complement of $\cR$, that exists since $S = \R^d$ here.

\begin{theorem}[{\cite[Theorem 1]{kovalev2021adom}}] \label{thm:adom}
    Let  functions $h_i\colon \R^d \to \R$ be $L$ smooth and $\mu$ strongly convex, $\bx^*$ be the solution of the optimization problem \eqref{eq: opt prob for ADOM setup with bx^*}, $\bW_n$ be a communication matrix at the $n$-th iteration satisfying Assumption \ref{assumption-lambdas}.
	Set parameters $\alpha, \eta, \theta, \sigma,\tau$ of Algorithm~\ref{alg:adom} to $\alpha = \frac{1}{2L}$, $\eta = \frac{2\lminp\sqrt{\mu L}}{7\lmax}$, $\theta = \frac{\mu}{\lmax}$, $\sigma = \frac{1}{\lmax}$, and $\tau = \frac{\lminp}{7\lmax}\sqrt{\frac{\mu}{L}}$. Then there exists $C>0$, such that for Fenchel conjugate function $H^*(\bz)$ from~\eqref{eq-prob-H*}
	\begin{equation}
		\left\|\nabla H^*(\bz_g^n) - \bx^*\right\|^2_2 \leq C \left(1- \frac{\lminp}{7\lmax} \sqrt{\frac{\mu}{L}}\right)^n.
	\end{equation}
\end{theorem}
 \begin{remark}
 \label{rem: adom exact C}
 Addressing details of the proof of Theorem~1 of {\cite{kovalev2021adom}} we see that  there is a particular choice of the constant $C$, namely
 \begin{equation}
 \label{eq: rate C from adom}
     C = \max \left\{\frac{2\tau}{\mu^2}, \frac{\tau(1-\tau)L}{\eta(1-\eta \alpha)\mu^2} \right\} = \frac{1}{\mu^2}\max \left\{\frac{2\lminp\sqrt{\mu}}{7\lmax \sqrt{L}}, \frac{1}{2} \right\} = \frac{1}{2\mu^2}.
 \end{equation}
 It means that the actual convergence rate is $n = \cO \left(\frac{\lmax}{\lminp}\sqrt{\frac{L}{\mu}}\ln \frac{1}{\mu^2\varepsilon}\right)$.
 \end{remark}
\begin{algorithm}[H]
	\caption{ADOM: Accelerated Decentralized Optimization Method}
	\label{alg:adom}
	\begin{algorithmic}[1]
		\State {\bf input:} $\nabla H^*\colon (\R^d)^m\to \R$, $\bz^0 \in \cR^\perp, \boldm^0 \in (\R^d)^\cV, \alpha, \eta, \theta,\sigma\!>\!0, \tau\!\in\!(0,1)$
		\State set $\bz_f^0 = \bz^0$
		\For{$k = 0,1,2,\ldots$}
		\State $\bz_g^n = \tau \bz^n + (1-\tau)\bz_f^n$\label{dual:line:z1}
		\State $\Delta^n = \sigma\bW_n(\boldm^n - \eta\nabla H^*(\bz_g^n))$\label{dual:line:delta}
		\State $\boldm^{n+1} = \boldm^n - \eta\nabla H^*(\bz_g^n) - \Delta^n$\label{dual:line:m}
		\State $\bz^{n+1} = \bz^n + \eta\alpha(\bz_g^n - \bz^n) + \Delta^n$\label{dual:line:z2}
		\State $\bz_f^{n+1} = \bz_g^n - \theta\bW_n\nabla H^*(\bz_g^n)$\label{dual:line:z3}
		\EndFor
	\end{algorithmic}
\end{algorithm}

\section{Proof of Theorem \ref{theorem: strongly convex functions}}
\label{proof main theorem}
 
 All the arguments below are applied under assumptions of Theorem~\ref{theorem: strongly convex functions}, i.e. we assume that $S\subset \R^d$ is a convex set, $\bx\in \cS $ is equivalent to $[\bx]_i\in S$ for all $i=1,\ldots,m$, functions $f^{\gamma}_i\colon S\to \R$ are $\gamma$ strongly convex, and the output of Algorithm~\ref{alg:mod-adom} is $\bx^n_{r,\gamma} = \nabla (H^{r,\gamma})^*(\bz_g^n)$. Denote also
 \begin{align*}
     \bx_{\gamma}^* = (x^*_{\gamma},\ldots,x^*_{\gamma}) = \argmin\limits_{\bx\in \cS} F^{\gamma}(\bx) = \argmin\limits_{x\in S} \sum\limits_{i=1}^mf_i^{\gamma}(x).
 \end{align*}

\subsection{Derivation of $(H^{r,\gamma})^*$}
\label{subsect: derivation H*}

In brief, in this subsection we show that functions  $h_i^{r,\gamma}$ from \eqref{eq: proof: decl H} are $\frac{1}{r}$ smooth, $\frac{\gamma}{1+r\gamma}$ strongly convex, and such that $\nabla (H^{r,\gamma})^*$ from Line~\ref{eq: theorem declaration nablaH^*} of Algorithm \ref{alg:mod-adom} is the gradient of the conjugate function $(H^{r,\gamma})^*$ of $H^{r,\gamma} = \sum\limits_{i=1}^m h_i^{r,\gamma}$ from \eqref{eq: proof: decl H}. Then the consensus condition~\eqref{eq: theorem: consensus} becomes a corollary of Theorem \ref{thm:adom} with $L = \frac{1}{r}$ and $\mu = \frac{\gamma}{1+r\gamma}$.

From now on let functions $h_i^{r,\gamma}\colon \R^d \to \R$ and $H^{r,\gamma}\colon (\R^d)^m \to \R$ be
 \begin{equation}
 \label{eq: proof: decl H}
        \begin{array}{cc}
             H^{r,\gamma}(\bx) = \sum\limits_{i=1}^{m} h_i^{r,\gamma}([\bx]_i), \ \mbox{where}
             \\
             h_i^{r,\gamma}(x) = \inf\limits_{y\in S}\left\{f_i^{\gamma}(y) + \frac{1}{2r}\|y-x\|^2_2\right\}.
        \end{array}
    \end{equation}
Define their conjugate as $(h_i^{r,\gamma})^*$ and $(H^{r,\gamma})^*$.
\begin{lemma}
 If functions $h_i^{r,\gamma}$ and $H^{r,\gamma}$ are defined by \eqref{eq: proof: decl H}, then their Fenchel conjugate functions $(h_i^{r,\gamma})^*$ and $(H^{r,\gamma})^*\colon (\R^d)^m \to \R$ are
 \begin{equation*}
 \label{eq: proof: decl H*}
        \begin{array}{cc}
             (H^{r,\gamma})^*(\bz) = \sum\limits_{i=1}^{m} (h_i^{r,\gamma})^*([\bz]_i), \ \mbox{where}
             \\
             (h_i^{r,\gamma})^*(z) = (f_i^{\gamma})^*(z) + \frac{r}{2}\|z\|^2_2.
        \end{array}
    \end{equation*}
Moreover, its conjugate $(H^{r,\gamma})^{**}$ coincides with $H^{r,\gamma}$.
 \end{lemma}
 \begin{proof}
 The definition \eqref{eq: proof: decl H} is similar to Moreau--Yosida smoothing, but the tricky point is that the functions~$f_i^{\gamma}$ are defined on a convex set~$S$ instead of the $\R^d$. Let us introduce functions $\tilde{f}_i^{\gamma}$ with   domain $\R^d$ as follows:
 \begin{equation}
\label{eq tilde f}
    \tilde{f}_i^{\gamma}(x)=\left\{\begin{array}{cc} 
         &f_i^{\gamma}(x) \quad \mbox{ if $x\in S$}\\
         & +\infty, \quad\mbox{ otherwise}.
    \end{array}\right.
\end{equation}
Such $\tilde{f}_i^{\gamma}$ are $\gamma$ strongly convex as well. Moreover, substitution $\tilde{f}_i^{\gamma}$ for $f_i^{\gamma}$  affect neither primal $h_i^{r,\gamma}$:
\[h_i^{r,\gamma}(x)  = \inf\limits_{y\in S}\left\{f_i^{\gamma}(y) + \frac{1}{2r}\|y-x\|^2_2\right\}= \inf\limits_{y\in \R^d}\left\{\tilde{f}_i^{\gamma}(y) + \frac{1}{2r}\|y-x\|^2_2\right\},\]
nor $(f_i^{\gamma})^*(z) + \frac{r}{2}\|z\|^2_2$: 
\begin{align*}
(f^{\gamma}_i)^*(z) + \frac{r}{2}\|z\|^2_2=  \max\limits_{x \in S}\left\lbrace  \left\langle z, x\right\rangle  -  f^{\gamma}_i(x)\right\rbrace + \frac{r}{2}\|z\|^2_2 
\\
=\max\limits_{x \in \R^d}\left\lbrace  \left\langle z, x\right\rangle  - \tilde{f}^{\gamma}_i(x)\right\rbrace + \frac{r}{2}\|z\|^2_2 = (\tilde{f}^{\gamma}_i)^*(z)+ \frac{r}{2}\|z\|^2_2.
\end{align*}

For each $i$ one can see that $(h_i^{r,\gamma})^* = (f^{\gamma}_i)^*(z) + \frac{r}{2}\|z\|^2_2$ is the Fenchel conjugate of $h_i^{r,\gamma}$ and vice versa. Indeed, for proper, convex and lower semicontinuous $g_1, g_2\colon \R^d \to \R$ we have $(g_1+ g_2)^*(x) = g_1^* \square g_2^*$ and $(g_1 \square g_2)^* = g_1^* + g_2^*$, where $(g_1\square g_2)(x)$ means the convolution $\inf\{g_1(y) + g_2(x-y) \mid y\in \R^d \}$.

Hence the Fenchel conjugate for the function $H^{r,\gamma}$ will be 
\begin{gather}
    \sup\limits_{\bx\in (\R^d)^m} \left\{\langle \bz,\bx \rangle - H^{r,\gamma}(\bx)\right\} \notag
   \\
   = \sup\limits_{\bx\in (\R^d)^m} \left\{\sum\limits_{i=1}^{m} \left( \langle [\bz]_i,[\bx]_i \rangle - h_{i}^{r,\gamma}([\bx]_i)\right) \right\} \\
   =\sum\limits_{i=1}^{m} \sup\limits_{[\bx]_i\in \R^d} \left\{\langle [\bz]_i,[\bx]_i \rangle - h_{i}^{r,\gamma}([\bx]_i)\right\} \notag\\
   = \sum\limits_{i=1}^{m} (h_{i}^{r,\gamma})^*([\bz]_i) = (H^{r,\gamma})^*(\bz).\notag
 \end{gather}
  In the same way one can see that $H^{r,\gamma}$ and $(H^{r,\gamma})^{**}$ coincide.
 \end{proof}

\begin{remark}
\label{rem: our L mu as r gamma}
 For each $i$ the function $\left(h^{r,\gamma}_{i}\right)^*$ from \eqref{eq: proof: decl H} is $\left(\frac{1}{\gamma}+r\right)$ smooth and $r$ strongly convex by definition, so we have $h^{r,\gamma}_i = (h^{r,\gamma}_i)^{**}$ being $\frac{1}{r}$ smooth and $\frac{\gamma}{1+r\gamma}$ strongly convex. In addition 
 \[\nabla (h^{r,\gamma}_i)^*(z) = \nabla (f_i^{\gamma})^*(z) + z \] as stated in Line~\ref{eq: theorem declaration nablaH^*} of Algorithm \ref{alg:mod-adom}. Then we can apply Algorithm \ref{alg:adom} for $L = r^{-1}$ smooth and $\mu = \frac{\gamma}{1+r\gamma}$ strongly convex functions $h^{r,\gamma}_i$ and get the values of $\nabla (h^{r,\gamma}_i)^*(z)$ as output.
 \end{remark} 
 Thus we construct a relaxation  $\min_{\bx\in \cR}H^{r,\gamma}(\bx)$ of the constrained convex optimization problem $\min_{\bx\in \cS}F^{\gamma}(\bx)$.
 \begin{corollary}
 \label{cor: to apply adom to H}
 Let a function $H^{r,\gamma}$ be defined in \eqref{eq: proof: decl H} and let $\bx^*_{r,\gamma} = \argmin\limits_{\bx\in \cR}H^{r,\gamma}(\bx)$. Then applying Algorithm \ref{alg:adom} for \[\nabla (h^{r,\gamma}_i)^*(z) = (f_i^{\gamma})^*(z) + rz\] we get by Theorem \ref{thm:adom} 
 \begin{equation}
 \label{eq: cor: convergence in argument}
    \left\|\bx^{*}_{r,\gamma} - \bx^{n}_{r,\gamma}\right\|^2_2\leq C \left(1- \frac{\lminp}{7\lmax} \sqrt{\frac{r\gamma}{1+r\gamma}}\right)^n,
\end{equation} 
where $\bx_{r,\gamma}^n = \nabla (H^{r,\gamma})^*(\bz_g^n)$ 
and \[C =  \frac{(1+r\gamma)^2}{2\gamma^2}.\]
 Moreover, since $\bx^*_{r,\gamma} \in \cR$, i.e. $[\bx^*_{r,\gamma}]_i = [\bx^*_{r,\gamma}]_j$ for all $i$ and $j$, the consensus condition is approximated as follows
\begin{equation*}
    \left\|\left[\bx^{n}_{r,\gamma}\right]_i - \left[\bx^{n}_{r,\gamma}\right]_j\right\|^2_2\leq 2C \left(1- \frac{\lminp}{7\lmax} \sqrt{\frac{r\gamma}{1+r\gamma}}\right)^n.
\end{equation*} 
 \end{corollary}

\subsection{Value bounds on $H^{r,\gamma}$}
\label{subsect: auxiliary properties}

Despite we defined $h_i^{r,\gamma}$ for all $\R^d$, some properties hold true on the initial set $S$ only.
\begin{lemma}
Let functions $h_i^{r,\gamma}$ be defined in \eqref{eq: proof: decl H}. If $x\in S$, then for any $r>0$, for each $i = 1,\ldots,m$ we have
\begin{equation}
\begin{array}{rl}
\label{eq-bounds_on_g**}
f^{\gamma}_i(x) - \frac{r}{2(1+r\gamma)}\left\|\nabla f^{\gamma}_i(x)\right\|^2_2 \leq h_i^{r,\gamma}(x) \leq f^{\gamma}_i(x).
 \end{array}
\end{equation} 
\end{lemma}
\begin{proof}
The second inequality directly follows from the definition \eqref{eq: proof: decl H}. To prove the first one we recall that $f^{\gamma}_i$ is $\gamma$ strongly convex and the following holds:
\begin{gather*}
h_i^{r,\gamma}(x) =\inf\limits_{y \in S} \left\{f^{\gamma}_i(y) +  (2r)^{-1}\|x-y\|^2_2\right\} 
\\=\inf\limits_{y: \ (x-y)\in S} \left\{f^{\gamma}_i(x-y) +  (2r)^{-1}\|y\|^2_2\right\}
\\ \geq \inf\limits_{y: \ (x-y)\in S} \left\{f^{\gamma}_i(x) + \langle \nabla f^{\gamma}_i(x), -y \rangle + \gamma/2 \|y\|^2_2 +  (2r)^{-1}\|y\|^2_2\right\}\\
\geq \inf\limits_{y\in \R^d} \left\{f^{\gamma}_i(x) + \langle \nabla f^{\gamma}_i(x), -y \rangle + \gamma/2 \|y\|^2_2 +  (2r)^{-1}\|y\|^2_2\right\},
\end{gather*}
which reaches its minimum at $y=\frac{r}{1+r\gamma}\nabla f^{\gamma}_i(x)$ and so equals to
\begin{gather*}
     f^{\gamma}_i(x) + \frac{r}{1+r\gamma} \langle -\nabla f^{\gamma}_i(x), \nabla f^{\gamma}_i(x) \rangle + \frac{r}{2(1+r\gamma)}\|\nabla f^{\gamma}_i(x)\|^2_2
    \\=f^{\gamma}_i(x) - \frac{r}{2(1+r\gamma)}\|\nabla f^{\gamma}_i(x)\|^2_2.
\end{gather*}
\end{proof}

\subsection{Convergence in argument}

Lemma \ref{lem: dist betweem argmins} shows convergence in argument in the following sense: if the regularization parameter~$r$ tends to zero,  the argminimum $\bx^*\rgam\in \cR$ of $H^{r,\gamma}$ tends to the argminimum $\bx^*_{\gamma}\in \cS$ of $F^{\gamma}$. By Corollary~\ref{cor: to apply adom to H} we have $\bx^*\rgam\in \cR$ approximated by $\bx^n\rgam\in (\R^d )^m$ for a sufficient  number of iterations~$n$.

\begin{lemma}
\label{lem: dist betweem argmins}
 Let $\bx^*_{r,\gamma} = \argmin_{\bx\in \cR} H^{r,\gamma}(\bx)$ for $H^{r,\gamma}$ defined in \eqref{eq: proof: decl H}. Let
\begin{equation}
    \label{eq: assumption for K-xi and nablas}
    \|\nabla F^{\gamma}(\bx)\|_2^2 \leq m K_{\zeta}^2  \quad  \forall \bx \in \{\by\in \cS \mid \|\by -\bx^*_{\gamma}\|_2\leq \zeta\}.
    \end{equation}
If $r$ is such that $\|\bx^*\rgam - \bx^*_{\gamma}\|_2\leq \zeta$, then
 \begin{equation}
    \|\bx^*\rgam - \bx^*_{\gamma}\|_2\leq \sqrt{\frac{rm}{2\gamma}}K_{\zeta}.
\end{equation}
\end{lemma}
 \begin{proof}
 Using \eqref{eq-bounds_on_g**} and strong convexity of $F^{\gamma}$ and $H^{r,\gamma}$ we have
 \begin{gather*}
     F^{\gamma} (\bx^*_{\gamma}) \geq H^{r,\gamma}(\bx^*_{\gamma}) = \sum h_i^{r,\gamma}([\bx^*_{\gamma}]_i) 
     \\
     \geq \sum\limits_{i=1}^m \left(h_i^{r,\gamma}(\bx^*_{r,\gamma}) + \frac{\gamma}{2(1+r\gamma)}\|[\bx^*\rgam]_i - [\bx^*_{\gamma}]_i\|_2^2\right) 
     \\
     = H^{r,\gamma}(\bx^*_{r,\gamma}) + \frac{\gamma}{2(1+r\gamma)}\|\bx^*\rgam - \bx^*_{\gamma}\|_2^2
     \\
     \geq F^{\gamma}(\bx^*\rgam) - \frac{r}{2(1+r\gamma)}\|\nabla F^{\gamma}(\bx^*\rgam)\|^2_2 + \frac{\gamma}{2(1+r\gamma)}\|\bx^*\rgam - \bx^*_{\gamma}\|_2^2
     \\
     \geq F^{\gamma}(\bx^*\rgam) - \frac{r}{2(1+r\gamma)}mK_{\zeta}^2 + \frac{\gamma}{2(1+r\gamma)}\|\bx^*\rgam - \bx^*_{\gamma}\|_2^2
     \\
     \geq F^{\gamma}(\bx^*_{\gamma}) + \gamma/2 \|\bx^*\rgam - \bx^*_{\gamma}\|_2^2 - \frac{r}{2(1+r\gamma)}mK_{\zeta}^2 + \frac{\gamma}{2(1+r\gamma)}\|\bx^*\rgam - \bx^*_{\gamma}\|_2^2
     \\
     \geq F^{\gamma}(\bx^*_{\gamma}) + \frac{\gamma}{1+r\gamma} \|\bx^*\rgam - \bx^*_{\gamma}\|_2^2 - \frac{r}{2(1+r\gamma)}mK_{\zeta}^2.
 \end{gather*}
 Then $\frac{\gamma}{1+r\gamma} \|\bx^*\rgam - \bx^*_{\gamma}\|_2^2 - \frac{r}{2(1+r\gamma)}mK_{\zeta}^2\leq 0$ and hence
 $\|\bx^*\rgam - \bx^*_{\gamma}\|^2_2 \leq \frac{rm}{2\gamma}K^2_{\zeta}$.
\end{proof}

Combining Lemma \ref{lem: dist betweem argmins} with Corollary \ref{cor: to apply adom to H} we get the following.
\begin{remark}
\label{rem: cond on zeta}
Let $\zeta>0$ and let $K_{\zeta}$ be such that \eqref{eq: assumption for K-xi and nablas} holds.
If
\begin{align*}
    \sqrt{\frac{rm}{2\gamma}}K_{\zeta}+\sqrt{C_1}\left(1-\frac{\lminp}{7\lmax}\sqrt{\frac{r\gamma}{1+r\gamma}}\right)^{n/2}\leq\zeta,
\end{align*}
where $C_1 = \frac{(1+r\gamma)^2}{2\gamma^2}$, then both %
$\|\bx^*\rgam - \bx^*_{\gamma}\|_2\leq \zeta$ and $\|\bx^n\rgam - \bx^*_{\gamma}\|_2\leq \zeta$ hold.
\end{remark}

\subsection{Value approximation}

Let $\bx^*_{r,\gamma}\in \cR$ be the only argminimum of $H^{r,\gamma}$ on the consensus space $\cR$, i.e.
\begin{equation}
    \begin{array}{cc}
         \bx^*_{r,\gamma} = \argmin\limits_{\bx\in \cS} H^{r,\gamma}(\bx).
    \end{array}
\end{equation}
In order to prove the value approximation \eqref{eq: theorem: value approximation} let us  separate it into parts and estimate each of them:
\begin{subequations}
\begin{gather}
    F^{\gamma}(\bx^{n}_{r,\gamma}) - F^{\gamma}(\bx^*_{\gamma})
        \\ \leq  F^{\gamma}(\bx^n_{r,\gamma}) - H^{r,\gamma}(\bx^{n}_{r,\gamma})  \label{eq: separate 1}
        \\ +  H^{r,\gamma}(\bx^n_{r,\gamma}) - H^{r,\gamma}(\bx^{*}_{r,\gamma}) \label{eq: separate 2}
        \\ +   H^{r,\gamma}(\bx^{*}_{r,\gamma}) - F^{\gamma}(\bx^{*}_{\gamma}). \label{eq: separate 3}
\end{gather}
\end{subequations}
The last addend is negative and can be eliminated:
\begin{equation*}
     H^{r,\gamma}(\bx^{*}_{r,\gamma}) - F^{\gamma}(\bx^{*}_{\gamma}) \leq  H^{r,\gamma}(\bx^{*}_{\gamma}) - F^{\gamma}(\bx^{*}_{\gamma}) \leq 0.
\end{equation*}
The rest are estimated in Lemmas \ref{lem: addend 1} and \ref{lem: c2 bound} under additional assumptions.
\begin{lemma}
\label{lem: addend 1}
Let $\|\bx^n\rgam - \bx^*_{\gamma}\|_2\leq \zeta$. If \eqref{eq: assumption for K-xi and nablas} holds, then
\begin{align}
\label{eq: from lemma}
        F^{\gamma}(\bx^n_{r,\gamma}) - H^{r,\gamma}(\bx^{n}_{r,\gamma})  \leq \frac{r}{2(1+r\gamma)}mK^2_{\zeta}.
\end{align}
\end{lemma}
\begin{proof}
We cannot declare a uniform $K$ instead of $K_{\zeta}$ because $F^{\gamma}$ is not smooth. Nonetheless, assuming %
$\bx^{n}_{r,\gamma}$ belong to $\zeta$-neighborhood of $\bx^*_{\gamma}$,  we immediately obtain  from \eqref{eq-bounds_on_g**} and \eqref{eq: assumption for K-xi and nablas} that
\begin{equation*}
    F^{\gamma}(\bx^n_{r,\gamma}) - H^{r,\gamma}(\bx^{n}_{r,\gamma}) \leq \frac{r}{2(1+r\gamma)}\|\nabla F^{\gamma}(\bx^{n}_{r,\gamma})\|^2_2 \leq \frac{r}{2(1+r\gamma)}mK^2_{\zeta}.
\end{equation*}
\end{proof}
\begin{lemma}
\label{lem: c2 bound}
Let \eqref{eq: assumption for K-xi and nablas} holds.
Then
\begin{align*}
H^{r,\gamma}(\bx^{n}_{r,\gamma}) - H^{r,\gamma}(\bx^*_{r,\gamma}) 
&\leq C_2 \left( 1 - \frac{\lminp}{7\lmax} \sqrt{\frac{r\gamma}{1+r\gamma}}\right)^{n/2},
\\
\mbox{where }\quad C_2 &= 
    \frac{m(1+r\gamma)K_{\zeta}}{\sqrt{2}\gamma}\sqrt{\frac{\lmax}{\lminp}} 
    + \frac{m(1+r\gamma)^2}{4r\gamma^2}.
\end{align*}
\end{lemma}
\begin{proof}
By $\frac{m}{r}$ smoothness of $H^{r,\gamma}$
\begin{equation*}
\begin{array}{rl}
&H^{r,\gamma}(\bx^{n}_{r,\gamma}) - H^{r,\gamma}(\bx^*_{r,\gamma})
\leq
\langle\nabla H^{r,\gamma}\left(\bx^*_{r,\gamma}\right),\bx^{n}_{r,\gamma} - \bx^{*}_{r,\gamma}\rangle  + \frac{m}{2r}\|\bx^n_{r,\gamma} - \bx^*_{r,\gamma}\|^2_2
\\
&\leq \langle\nabla H^{r,\gamma}\left(\nabla (H^{r,\gamma})^*(\bz_g^{\infty})\right),\nabla (H^{r,\gamma})^*(\bz_g^{n}) - \bx^{*}_{r,\gamma}\rangle 
+ \frac{m}{2r}\|\bx^n_{r,\gamma} - \bx^*_{r,\gamma}\|^2_2
\\
&\leq \langle \bz_g^{\infty}, \nabla (H^{r,\gamma})^*(\bz_g^{n}) - \bx^{*}_{r,\gamma}\rangle + \frac{m}{2r}\|\bx^n_{r,\gamma} - \bx^*_{r,\gamma}\|^2_2,
\end{array}
\end{equation*}
where $\bz^{\infty}_{g}$ is the limit of $\bz_g^n$ and so it is the argminimum of $(H^{r,\gamma})^*$ on $\cRperp$.
By~\eqref{eq: cor: convergence in argument} we have 
\[\frac{m}{2r}\|\bx^n\rgam - \bx^*\rgam\|^2_2
\leq \frac{m}{2r}C_1\left( 1 - \frac{\lminp}{7\lmax} \sqrt{\frac{r\gamma}{1+r\gamma}}\right)^n 
= \frac{m(1+r\gamma)^2}{4r\gamma^2}\left( 1 - \frac{\lminp}{7\lmax} \sqrt{\frac{r\gamma}{1+r\gamma}}\right)^n\]

Let us introduce an orthogonal projection matrix $\bP$ onto the subspace $\cRperp$, i.e., it holds $\bP v = \argmin_{z \in \cRperp} \{v - z\}$
for an arbitrary $v \in (\R^d)^n$. Then matrix $\bP$ is 
\begin{equation}\label{eq:P}
 \bP = \left(\bI_n - \frac{1}{n}\boldOne_n\boldOne_n\transpose\right)\otimes \bI_d,
\end{equation}
where $\bI_n$ denotes $n\times n$ identity matrix, $\boldOne_n = (1,\ldots,1)\in \R^n$, and $\otimes$ is a Kronecker product. Note that
$\bP\transpose\bP = \bP$.

Since $\bz_g^{\infty}\in \cRperp$ and  $\bx^{*}_{r,\gamma}\in \cR$, the first part simplifies to $\langle \bz_g^{\infty}, \bP\nabla (H^{r,\gamma})^*(\bz_g^{n}) \rangle$. 
We may use Lemma 2 in \cite{kovalev2021adom} to get the following estimation 
\begin{equation*}
\begin{array}{rl}
\|\bP \nabla (H^{r,\gamma})^*(\bz_g^{n})\|_2^2 = \|\nabla (H^{r,\gamma})^*(\bz_g^{n})\|_{\bP}^2 
\leq \frac{2}{\theta \lminp}\left((H^{r,\gamma})^*(\bz^n_g) - (H^{r,\gamma})^*(\bz^{n+1}_f)\right).
\end{array}
\end{equation*}
As $\bz^{n+1}_f$ is a non-optimal point of Algorithm~\ref{alg:mod-adom}, this is not greater than
\begin{gather*}
\frac{2}{\theta \lminp}\left((H^{r,\gamma})^*(\bz^n_g) - (H^{r,\gamma})^*(\bz^*)\right)
\\ 
\leq 
\frac{m(1+r\gamma)}{\gamma\theta \lminp}\left\|\bz^n_g - \bz^*\right\|_2^2 
=
\frac{m(1+r\gamma)^2}{\gamma^2}\frac{\lmax}{\lminp}\left\|\bz^n_g - \bz^*\right\|_2^2 
\\
\leq 
\frac{m(1+r\gamma)^2}{2\gamma^2}\frac{\lmax}{\lminp}
\left( 1 - \frac{\lminp}{7\lmax} \sqrt{\frac{r\gamma}{1+r\gamma}}\right)^n
\end{gather*}
and the latter ones follow from the $\frac{m(1+r\gamma)}{\gamma}$ smoothness of $(H^{r,\gamma})^*$ and from the fact that the proof of \cite[Theorem 1]{kovalev2021adom} actually covers the following chain of inequalities:
 \begin{equation*}
     \left\|\nabla H^*(\bz_g^n) - \bx^*\right\|^2_2 \leq \frac{1}{\mu^2}\left\|\bz^n_g - \bz^* \right\|^2_2 \leq C \left( 1 - \frac{\lminp}{7\lmax}\sqrt{\frac{\mu}{L}}\right)^n = \frac{1}{2\mu^2} \left( 1 - \frac{\lminp}{7\lmax}\sqrt{\frac{\mu}{L}}\right)^n.
 \end{equation*}

By our assumption $\|\bz^{\infty}_{g}\|_2 = \|\nabla H^{r,\gamma}(x\rgam^*)\|_2< \sqrt{m}K_{\zeta}$. 
Thus, we obtain 
\begin{equation*}
\begin{array}{rl}
     (H^{r,\gamma})^*(\bx^n_{r,\gamma}) - (H^{r,\gamma})^*(\bx^*_{r,\gamma})
     \\ \leq
         \sqrt{m}K_{\zeta}\frac{\sqrt{m}(1+r\gamma)}{\sqrt{2}\gamma}\sqrt{\frac{\lmax}{\lminp}}
         \left( 1 - \frac{\lminp}{7\lmax}\sqrt{\frac{r\gamma}{1+r\gamma}}\right)^{n/2} + \frac{m(1+r\gamma)^2}{4r\gamma^2}\left( 1 - \frac{\lminp}{7\lmax} \sqrt{\frac{r\gamma}{1+r\gamma}}\right)^n
    \\ \leq \left(\frac{m(1+r\gamma)K_{\zeta}}{\sqrt{2}\gamma}\sqrt{\frac{\lmax}{\lminp}} 
    + \frac{m(1+r\gamma)^2}{4r\gamma^2}\right)
    \left( 1 - \frac{\lminp}{7\lmax}\sqrt{\frac{r\gamma}{1+r\gamma}}\right)^{n/2}. 
    \\ = C_2 \left( 1 - \frac{\lminp}{7\lmax}\sqrt{\frac{r\gamma}{1+r\gamma}}\right)^{n/2}.
\end{array}
\end{equation*}
\end{proof}

\subsection{Final compilation}
\label{proof numb of iterations}

This section completes the proof of Theorem~\ref{theorem: strongly convex functions} and shows Remark \ref{rem: numb of iterations}.

Recall that where $C_1 = \frac{(1+r\gamma)^2}{2\gamma}$ and $$C_2 = \frac{m}{2r}C_1 + \frac{m(1+r\gamma)K_{\zeta}}{\sqrt{2}\gamma}\frac{\lmax}{\lminp} = \frac{m(1+r\gamma)^2}{4r\gamma} + \frac{m(1+r\gamma)K_{\zeta}}{\sqrt{2}\gamma}\frac{\lmax}{\lminp}.$$ By Remark \ref{rem: cond on zeta} and Lemmas \ref{lem: addend 1}, \ref{lem: c2 bound} we see that $F^{\gamma}(\bx^{n}_{r,\gamma}) - F^{\gamma}(\bx^*_{\gamma})< \varepsilon$ if
\begin{align}
\forall \bx\in \{\by\in  \mid \|\by-\bx_{\gamma}^*\|_2<\zeta\} \qquad \|\nabla F^{\gamma}(\bx)\|_2^2&<mK^2_{\zeta}, \label{eq: choice of K}
\\
    \label{eq: valappr cond: xi}
    \sqrt{\frac{rm}{2\gamma}}K_{\zeta}+\sqrt{C_1}\left(1-\frac{\lminp}{7\lmax}\sqrt{\frac{r\gamma}{1+r\gamma}}\right)^{n/2}&\leq \zeta,
    \\
\frac{r}{2(1+r\gamma)}mK^2_{\zeta} &\leq  \varepsilon/2, \label{eq: valappr cond: epsilon-1}
\\
C_2 \left(1 - \frac{\lminp}{7\lmax} \sqrt{\frac{r\gamma}{1+r\gamma}}\right)^{n/2}
&\leq  \varepsilon/2. \label{eq: valappr cond: epsilon-2}
\end{align}

Let $\zeta =\sqrt{\varepsilon / \gamma}$ and let   $r\leq \frac{\varepsilon}{2mK^2_{\zeta}}$. Then \eqref{eq: valappr cond: epsilon-1} holds. If \eqref{eq: valappr cond: epsilon-2} fulfills, then \eqref{eq: valappr cond: xi} follows from \eqref{eq: valappr cond: epsilon-1} and \eqref{eq: valappr cond: epsilon-2} as $\sqrt{\frac{rm}{2\gamma}}K_{\zeta}\leq \sqrt{\frac{\varepsilon}{2\gamma}} \leq \zeta/\sqrt{2}$  and  $\sqrt{C_1}\left(1-\frac{\lminp}{7\lmax}\sqrt{\frac{r\gamma}{1+r\gamma}}\right)^{n/2}\leq \zeta/2$ since $1\leq \sqrt{C_1}\leq C_1\leq C_2$ and $\varepsilon\leq \sqrt{\varepsilon/\gamma} = \zeta$. 
Thus, it suffices to assume
\begin{align*}
    \forall i \quad \forall x\in \lbrace y\in S \mid \|y-x_{\gamma}^*\|_2^2\leq \varepsilon/ \gamma \rbrace \qquad \|\nabla f_i^{\gamma}(x)\|_2&\leq K,
    \\
    r&\leq \frac{\varepsilon}{2mK^2}, 
    \\
    C_2
\left(1 - \frac{\lminp}{7\lmax} \sqrt{\frac{r\gamma}{1+r\gamma}}\right)^{n/2}
&\leq \varepsilon/2.
\end{align*}
So $\varepsilon$ approximation requires a number of iteration $$\cO\left(\frac{\lmax}{\lminp}\sqrt{\frac{1+r\gamma}{r\gamma}}\ln \frac{C_2}{\varepsilon}\right) = \cO\left(\frac{\lmax}{\lminp}\frac{1}{\sqrt{\gamma\varepsilon}}\ln \frac{1}{\varepsilon}\right).$$

\section{Proof of Theorem \ref{th: wasserstein}}
\label{proof wb throrem}

To prove Theorem \ref{th: wasserstein} we combine proved Theorem \ref{theorem: strongly convex functions} with features of the entropy regularization of the Wasserstein barycenter problem.

\subsection{Entropy regularized WB problem}
\label{sec:Entropy regularized WB problem}

Recall that for a fixed cost matrix $M$  we define the set of {\it transport plans}  as
		\begin{align*}
		U(p,q)  := \left\lbrace X \in \mathbb{R}_+^{d \times d} \mid X \boldOne = p, X^T\boldOne = q \right\rbrace
		\end{align*}
and {\it Wasserstein distance} between two probability distributions $p$ and $q$ as
\[\cW (p,q)  := \min_{X \in U(p,q)}  \langle M, X \rangle. \]
The entropy regularized (or smoothed) Wasserstein distance is defined as 
	\begin{align}
	\label{eq:wass_distance}
		\cW_{\gamma} (p,q)  := \min_{X \in U(p,q)} \left\lbrace  \langle M, X \rangle - \gamma E(X)\right\rbrace,
		\end{align}
where $\gamma>0$ and
	\begin{align}
		E(X) := - \sum_{i=1}^{d} \sum_{j=1}^{d} e(X_{ij}), \nonumber\\ \text{where }
		e(x)=\left\{
		\begin{array}{ll}
		     x\ln x \quad & \mbox{if } x>0 \\
		     0 \quad & \mbox{if } x=0.
		\end{array}\right.
		\end{align}
So it seeks to minimize the transportation costs while maximizing the entropy. 
Moreover $\cW_{\gamma}(p,q)\to \cW(p,q)$ as $\gamma\to 0$.

 Then the convex optimization problem \eqref{w_barycenter} can be relaxed to the following  $\gamma$ strongly convex optimization problem %
		\begin{align}
		\label{w_gamma_barycenter}
			 \min_{p \in \simplex}\sum\limits_{i=1}^{m} \cW_{\gamma,q_i}(p),
			\end{align}
where $\cW_{\gamma,q_i}(p) = \cW_{\gamma}(q_i, p)$. The argminimum of \eqref{w_gamma_barycenter} is called the uniform Wasserstein barycenter~\cite{agueh2011barycenters,cuturi2014fast} of the family of  $q_1,\ldots, q_m$.
Moreover, problem \eqref{w_gamma_barycenter} admits a unique solution and approximates unregularized WB problem as follows.
\begin{remark}
\label{rem: entropy approximation}
Let $\gamma\leq\frac{\varepsilon}{4}\ln d$.
If vectors $\hat{p}_i\in\simplex$ are such that 
\[\sum\limits_{i=1}^{m}\cW_{\gamma,q_i}(\hat{p}_i) - \min\limits_{p\in \simplex}\sum\limits_{i=1}^{m}\cW_{\gamma,q_i}(p) \leq \frac{\varepsilon}{2},\]
 then
\[\sum\limits_{i=1}^{m}\cW_{q_i}(\hat{p}_i) - \min\limits_{p\in \simplex}\sum\limits_{i=1}^{m}\cW_{q_i}(p)\leq  \varepsilon.\]
\end{remark}
 Indeed, as entropy is bounded we have $\cW_{q_i}(p)\leq \cW_{\gamma, q_i}(p)\leq \cW_{q_i}(p) + 2\gamma  \ln d$  for all $i$ and $p$. Then,  
 for $p^* = \argmin\limits_{p\in\simplex}\sum\limits_{i=1}^{m}\cW_{q_i}(p) $ and $p^*_{\gamma} = \argmin\limits_{p\in\simplex}\sum\limits_{i=1}^{m}\cW_{\gamma, q_i}(p) $ it holds that 
\begin{align*}
     \sum\limits_{i=1}^{m}\cW_{q_i}(\hat{p}_i) -  \sum\limits_{i=1}^{m}\cW_{q_i}(p^*) 
    \\ \leq
     \sum\limits_{i=1}^{m}\cW_{\gamma, q_i}(\hat{p}) -  \sum\limits_{i=1}^{m}\cW_{\gamma, q_i}(p^*) + 2\gamma \ln d 
    \\
    \leq  \sum\limits_{i=1}^{m}\cW_{\gamma, q_i}(\hat{p}) -  \sum\limits_{i=1}^{m}\cW_{\gamma, q_i}(p^*_{\gamma}) + \frac{\varepsilon}{2} \leq \varepsilon.
\end{align*}

\subsection{Legendre transforms}
One particular advantage of entropy regularization of the Wasserstein distance is that it yields closed-form representations for the dual function $\cW^*_{\gamma, q}(\cdot)$ and for its gradient. %
Recall that  the Fenchel-Legendre transform of~\eqref{eq:wass_distance} is defined as
		\begin{align}\label{eq:dual_wass}
		\cW^*_{\gamma,q}(z) & := \max_{p \in \simplex}\left\lbrace  \left\langle z,p\right\rangle  - {\it \cW_{\gamma,q} (p)}\right\rbrace .
		\end{align}
	\begin{theorem}[{\cite{cuturi2015smoothed}[Theorem 2.4]}]\label{thm:cuturi}
		For $\gamma >0$, the Fenchel-Legendre dual function $\cW^*_{\gamma,q}(z)$ is differentiable
		\begin{equation}
		\label{eq-W*}
		    \begin{array}{rl}
		    \cW^*_{\gamma,q}(z) & = \gamma\left(E(q) + \left\langle q, \ln \cK \alpha \right\rangle  \right)  
			\\&= - \gamma\left\langle q,\ln q \right\rangle +  \gamma\sum\limits_{j=1}^{m} [q]_j \ln\left(\sum\limits_{i=1}^{m}\exp \left(\frac{1}{\gamma}\left([z]_i-M_{ji}\right)\right)\right)  
		    \end{array}
		    \end{equation}
		    and its gradient $\nabla \cW^*_{\gamma,q}(z)$ is $1/\gamma$-Lipschitz {\it in the 2-norm} with
		\begin{equation}
		\label{eq-grad W*}
		    \begin{array}{rl}
		    \nabla \cW^*_{\gamma,q}(z)  &= \alpha \circ \left(\cK \cdot {q}/({\cK \alpha})\right) \in \simplex, %
			\\ \left[ \nabla \cW^*_{\gamma,q}(z) \right]_l & =\sum\limits_{j=1}^{m} [q]_j \frac{\exp\left(\frac{1}{\gamma}([z]_l - M_{lj})\right)}{\sum\limits_{i=1}^{m}\exp\left(\frac{1}{\gamma}([z]_i - M_{ij})\right)}.
		    \end{array}
	\end{equation}
		where $z \in \mathbb{R}^n$ and for brevity we denote $\alpha = \exp( {z}/{\gamma}) $ and $\cK = \exp\left( {-M}/{\gamma }\right)$.
	\end{theorem}
Notice that to get back and obtain the approximated barycenter we can employ the following result (with $\lambda_i = 1$).
\begin{theorem}[{\cite{cuturi2015smoothed}[Theorem 3.1]}] \label{th-barycenter-as-dual-gradient}
    The barycenter $p^*$ solving \eqref{w_gamma_barycenter} satisfies
\begin{equation*}
    \forall i=1,\ldots,m \qquad p^* =\nabla\cW^*_{\gamma,q_i}(z^*_i), 
\end{equation*}
where the set of $z^*_i$ constitutes any solution of any smoothed dual WB problem:
\begin{equation*}
    \min\limits_{z_1,\ldots,z_m\in \R^d} \sum\limits_{i=1}^{m}\lambda_i\cW^*_{\gamma,q_i}(z_i) \quad \mbox{s.t.} \quad \sum\limits_{i=1}^{m}\lambda_i z_i = 0.
\end{equation*}
\end{theorem}	

Thus we can apply Theorem \ref{theorem: strongly convex functions} for the problem~\eqref{w_gamma_barycenter} with explicitly defined $\nabla \cW^*_{\gamma, q_i}$ and obtain $\bx^n_{r,\gamma}$ that satisfies
\begin{align*}
 \sum\limits_{i=1}^{m}\cW_{\gamma,q_i}([\bx^n_{r,\gamma}]_i) - \min\limits_{p\in \simplex} \sum\limits_{i=1}^{m}\cW_{\gamma,q_i}(p) 
\\
\leq \frac{r}{4(1+r\gamma)}mK^2 + \frac{1}{2}C_2\left(1 - \frac{\lminp}{7\lmax}\sqrt{\frac{r\gamma}{1+r\gamma}}\right)^{n/2}\leq \varepsilon/2.    
\end{align*}
By Remark \ref{rem: entropy approximation} it proves \begin{align*}
    \left|\sum\limits_{i=1}^m \cW_{q_i}([\bx^{n}_{r,\gamma}]_i) - \sum\limits_{i=1}^m \cW_{q_i}([\bp^*]_i)\right|
    \\
    \leq
    2\gamma \ln d + \frac{r}{4(1+r\gamma)}mK^2 + C \left(1- \frac{\lminp}{7\lmax} \sqrt{\frac{r\gamma}{1+r\gamma}}\right)^{n}/2\leq \varepsilon,
\end{align*} for 
$C=\frac{1}{2}C_2 
= \frac{(1+r\gamma)mK_{\zeta}}{2\sqrt{2}\gamma}\sqrt{\frac{\lmax}{\lminp}}
    + \frac{(1+r\gamma)^2}{8r\gamma^2}$.

\subsection{Parameter estimation}

It remains to assign $\zeta>0$ and $K = K_{\zeta}$ satisfying \eqref{eq: choice of K}. Due to  Assumption \ref{assumption} such $\zeta$ and $K$ exist. 
\begin{proposition}
\label{prop-uniform_bound_on_grad_W}
Let a set $\{q_i\}_{i=1}^m$ satisfies Assumption \ref{assumption}, let $p^*_{\gamma}$ be the uniform Wasserstein barycenter of $\{q_i\}_{i=1}^m$, and let $\zeta \in \left(0, \min\{\frac{1}{e}, \ \min_{i,l} [q_i]_l\}\right)$. For each $i = 1,\ldots,m$  the norm of the gradient  $\|\nabla\cW_{\gamma,q_i}(\cdot)\|_2^2$  is uniformly bounded over $\{p\in \simplex \mid \|p-p^*_{\gamma} \|_2^2 \leq \zeta\};$ and the bound $K_{\rho}$ is given in \eqref{eq-bound_K_rho} for $\rho \leq   \min\{\frac{1}{e}, \ \min_{i,l} [q_i]_l\}-\zeta.$
\end{proposition}

We obtain  Proposition \ref{prop-uniform_bound_on_grad_W} as a combination of Lemma \ref{lem-bound_K_rho}  from \cite{bigot2019data} and proved below Lemma \ref{lem-p*_is_bounded_away_from_zero}.

\begin{lemma}[{\cite[Lemma 3.5]{bigot2019data}}]
\label{lem-bound_K_rho}
    For any $\rho\in (0,1)$, $q\in \simplex$, and $p\in \{x\in \simplex \mid \min_l x_l\geq \rho \}$ there is a bound: $\|\nabla \cW_{\gamma, q}(p) \|^2_2\leq K_{\rho}$, where
    \begin{equation}
        \label{eq-bound_K_rho}
        K_{\rho} = \sum\limits_{j=1}^{d}\left( 2\gamma\ln d + \inf_i\sup_l |M_{jl} - M_{il}| - \gamma\ln \rho \right)^2.
    \end{equation}
\end{lemma}

\begin{lemma}
\label{lem-p*_is_bounded_away_from_zero}
Let a set $\{q_i\}_{i=1}^m$ satisfies Assumption \ref{assumption}, let $p^*_{\gamma}$ be the uniform Wasserstein barycenter of~$\{q_i\}_{i=1}^m$.  All components $k$ of $p^*_{\gamma}$ have a uniform positive lower bound: $[p^*_{\gamma}]_k\geq \min\{\frac{1}{e}, \ \min_{i,l} [q_i]_l\}$. 
\end{lemma}
\begin{proof}
Let $X^*_i$ denote the optimal transport plan between $p^*_{\gamma}$ and $q_i$. Assume the contrary: there is $k$ such that $[p^*_{\gamma}]_k < \min\{\frac{1}{e}, \ \min_{i,l} [q_i]_l\}$. Then there is another component $n$ such that $[p^*_{\gamma}]_n>\min_i [q_i]_n > \min_{i,l} [q_i]_l$. 
Consider the vector $p$ that consists of $[p]_i = [p^*_{\gamma}]_i$ except for the components $[p]_n = [p^*_{\gamma}]_n+\delta$ and $[p]_l = [p^*_{\gamma}]_l-\delta$, where $\delta>0$ is less than $\min_{i,a\not=b}[X_i^*]_{a,b}$ of the optimal transport plans $X^*_i$ between $p^*_{\gamma}$ and $q_n$. Because of the entropy, all these optimal transport plans contain only positive non-diagonal elements, so such a $\delta$ exists. 

Construct now non-optimal transport plans between $p$ and each of $q_i$ in order to get the contradiction with the assumption.  Initially we have $\cW_{\gamma,q_i}(p^*_{\gamma}) = \langle C, X^*_i \rangle - \gamma X^*_i \ln X^*_i$. 
Consider the matrix $X_i$ that differs from $X^*_i$ only at four elements: 
\begin{align*}
          [X_i]_{kk} = [X_i^*]_{kk} +\frac{1}{2}\delta,  \quad [X_i]_{kn} = [X_i^*]_{kn} +\frac{1}{2}\delta,
        \\ [X_i]_{nn} = [X_i^*]_{nn} +\frac{1}{2}\delta, \quad  [X_i]_{nk} = [X_i^*]_{nk} +\frac{1}{2}\delta.
\end{align*}
Then $X_i$ is a transport plan between $p$ and $q_i$ since its elements are positive and also $X_i \boldOne = p$ and $X_i\transpose \boldOne = q_i$. Using the monotonicity of entropy on the interval $(0,\frac{1}{e})$ and the assumption that diagonal elements of the cost matrix $C$ are zero, we get  for each~$i$: 
\begin{equation*}
    \begin{array}{rl}
        \cW_{\gamma,q_i}(p) \leq & \langle C, X_i \rangle - \gamma X_i \ln X_i  
        \\ =&  \langle C, X^*_i \rangle - \gamma X^*_i \ln X^*_i  +\frac{1}{2}\delta C_{kn}-\frac{1}{2}\delta C_{nk} 
        \\ +& ([X_i]_{kk}\ln [X_i]_{kk} - [X^*_i]_{kk}\ln [X^*_i]_{kk})
        \\ +& ([X_i]_{kn}\ln [X_i]_{kn} - [X^*_i]_{kn}\ln [X^*_i]_{kn}) 
        \\ +& ([X_i]_{nk}\ln [X_i]_{nk} - [X^*_i]_{nk}\ln [X^*_i]_{nk}) 
        \\ +& ([X_i]_{nn}\ln [X_i]_{nn} - [X^*_i]_{nn}\ln [X^*_i]_{nn}) 
        \\ <&  \langle C, X^*_i \rangle - \gamma X^*_i \ln X^*_i  +\frac{1}{2}\delta C_{kn}-\frac{1}{2}\delta C_{nk}
        \\ =&  \langle C, X^*_i \rangle - \gamma X^*_i \ln X^*_i  = \cW_{\gamma,q_i}(p^*_{\gamma}).
    \end{array}
\end{equation*}
The obtained inequalities $\cW_{\gamma,q_i}(p)<\cW_{\gamma,q_i}(p^*_{\gamma})$ contradict to the fact that $p^*_{\gamma}$ is the barycenter; this proves the lemma. 
\end{proof}

\end{document}